\documentclass[12pt]{article}

\usepackage[T1]{fontenc}
\usepackage[margin=1in]{geometry}
\usepackage{microtype}
\usepackage{amsmath,amssymb,amsthm,mathtools}
\usepackage{newtxtext,newtxmath}
\usepackage{booktabs,longtable,tabularx}
\usepackage{algorithm}
\usepackage[noend]{algpseudocode}
\usepackage{graphicx}
\usepackage{tikz}
\usepackage{subcaption}
\usepackage{enumitem}
\usepackage{needspace}
\usepackage{listings}
\usepackage{xcolor}
\usepackage{xurl}
\usepackage{aliascnt}
\usepackage[colorlinks=true,linkcolor=blue!55!black,citecolor=blue!55!black,urlcolor=blue!55!black]{hyperref}
\makeatletter
\def\theHALG@line{\thealgorithm.\arabic{ALG@line}}
\makeatother

\definecolor{accent}{RGB}{36,78,120}
\definecolor{softgray}{RGB}{246,247,249}
\hypersetup{
  pdftitle={The Excluded Vertex-Minors and Pivot-Minors for Rank-Width at Most Two},
  pdfauthor={Sang-il Oum}
}
\setlist{itemsep=2pt,topsep=4pt}
\usepackage{keytheorems}
\usepackage{zref-clever}
\zcsetup{cap=true}
\newkeytheorem{theorem}[name=Theorem,refname={Theorem,Theorems}, Refname={Theorem,Theorems}]
\newkeytheorem{lemma,proposition,corollary,conjecture,problem}[sibling=theorem]
\newkeytheorem{definition}[sibling=theorem,style=definition]
\newkeytheorem{remark}[sibling=theorem,style=remark]
\newcommand{\Ftwo}{\mathbb{F}_2}
\newcommand{\rw}{\operatorname{rw}}
\newcommand{\bw}{\operatorname{bw}}
\newcommand{\rank}{\operatorname{rank}}
\newcommand{\cO}{\mathcal{O}}
\newcommand{\cF}{\mathcal{F}}
\newcommand{\cPiv}{\mathcal{P}^{\mathrm{piv}}}

\newcommand{\abs}[1]{\lvert #1\rvert}
\tikzset{
  obstruction edge/.style={draw=black,line width=0.4pt,line cap=round},
  obstruction vertex/.style={circle,draw,solid,fill=black,inner sep=0pt,minimum width=3pt}
}
\newcommand{\obstructiongraph}[1]{\csname obstructiongraph@#1\endcsname}

\expandafter\def\csname obstructiongraph@1\endcsname{\begin{tikzpicture}[x=1cm,y=1cm,baseline=(current bounding box.center)]
    \coordinate (g1v1) at (1.4002,-0.2363);
    \coordinate (g1v2) at (-0.0812,1.4177);
    \coordinate (g1v3) at (-1.3641,-0.3946);
    \coordinate (g1v4) at (-1.1590,0.8205);
    \coordinate (g1v5) at (1.0578,0.9474);
    \coordinate (g1v6) at (0.0000,0.0000);
    \coordinate (g1v7) at (0.6882,-1.2421);
    \coordinate (g1v8) at (-0.5420,-1.3125);
    \draw[obstruction edge] (g1v2) -- (g1v4);
    \draw[obstruction edge] (g1v3) -- (g1v4);
    \draw[obstruction edge] (g1v1) -- (g1v5);
    \draw[obstruction edge] (g1v2) -- (g1v5);
    \draw[obstruction edge] (g1v1) -- (g1v6);
    \draw[obstruction edge] (g1v2) -- (g1v6);
    \draw[obstruction edge] (g1v3) -- (g1v6);
    \draw[obstruction edge] (g1v4) -- (g1v6);
    \draw[obstruction edge] (g1v5) -- (g1v6);
    \draw[obstruction edge] (g1v1) -- (g1v7);
    \draw[obstruction edge] (g1v6) -- (g1v7);
    \draw[obstruction edge] (g1v3) -- (g1v8);
    \draw[obstruction edge] (g1v6) -- (g1v8);
    \draw[obstruction edge] (g1v7) -- (g1v8);
    \node[obstruction vertex] at (g1v1) {};
    \node[obstruction vertex] at (g1v2) {};
    \node[obstruction vertex] at (g1v3) {};
    \node[obstruction vertex] at (g1v4) {};
    \node[obstruction vertex] at (g1v5) {};
    \node[obstruction vertex] at (g1v6) {};
    \node[obstruction vertex] at (g1v7) {};
    \node[obstruction vertex] at (g1v8) {};
  \end{tikzpicture}}

\expandafter\def\csname obstructiongraph@2\endcsname{\begin{tikzpicture}[x=1cm,y=1cm,baseline=(current bounding box.center)]
    \coordinate (g2v1) at (0.0000,1.4200);
    \coordinate (g2v2) at (-1.2298,-0.7100);
    \coordinate (g2v3) at (1.2298,-0.7100);
    \coordinate (g2v4) at (0.3716,0.3346);
    \coordinate (g2v5) at (0.1039,-0.4891);
    \coordinate (g2v6) at (-0.4755,0.1545);
    \coordinate (g2v7) at (0.7134,0.6424);
    \coordinate (g2v8) at (0.1996,-0.9390);
    \coordinate (g2v9) at (-0.9130,0.2967);
    \draw[obstruction edge] (g2v1) -- (g2v2);
    \draw[obstruction edge] (g2v1) -- (g2v3);
    \draw[obstruction edge] (g2v2) -- (g2v3);
    \draw[obstruction edge] (g2v1) -- (g2v4);
    \draw[obstruction edge] (g2v3) -- (g2v4);
    \draw[obstruction edge] (g2v2) -- (g2v5);
    \draw[obstruction edge] (g2v3) -- (g2v5);
    \draw[obstruction edge] (g2v4) -- (g2v5);
    \draw[obstruction edge] (g2v1) -- (g2v6);
    \draw[obstruction edge] (g2v2) -- (g2v6);
    \draw[obstruction edge] (g2v4) -- (g2v6);
    \draw[obstruction edge] (g2v5) -- (g2v6);
    \draw[obstruction edge] (g2v1) -- (g2v7);
    \draw[obstruction edge] (g2v3) -- (g2v7);
    \draw[obstruction edge] (g2v4) -- (g2v7);
    \draw[obstruction edge] (g2v2) -- (g2v8);
    \draw[obstruction edge] (g2v3) -- (g2v8);
    \draw[obstruction edge] (g2v5) -- (g2v8);
    \draw[obstruction edge] (g2v7) -- (g2v8);
    \draw[obstruction edge] (g2v1) -- (g2v9);
    \draw[obstruction edge] (g2v2) -- (g2v9);
    \draw[obstruction edge] (g2v6) -- (g2v9);
    \draw[obstruction edge] (g2v7) -- (g2v9);
    \draw[obstruction edge] (g2v8) -- (g2v9);
    \node[obstruction vertex] at (g2v1) {};
    \node[obstruction vertex] at (g2v2) {};
    \node[obstruction vertex] at (g2v3) {};
    \node[obstruction vertex] at (g2v4) {};
    \node[obstruction vertex] at (g2v5) {};
    \node[obstruction vertex] at (g2v6) {};
    \node[obstruction vertex] at (g2v7) {};
    \node[obstruction vertex] at (g2v8) {};
    \node[obstruction vertex] at (g2v9) {};
  \end{tikzpicture}}

\expandafter\def\csname obstructiongraph@3\endcsname{\begin{tikzpicture}[x=1cm,y=1cm,baseline=(current bounding box.center)]
    \coordinate (g3v1) at (-0.9110,1.2050);
    \coordinate (g3v2) at (-0.4350,0.4860);
    \coordinate (g3v3) at (0.1470,0.1470);
    \coordinate (g3v4) at (0.1470,-1.1755);
    \coordinate (g3v5) at (-1.1755,0.1470);
    \coordinate (g3v6) at (-0.9110,-0.9110);
    \coordinate (g3v7) at (1.2050,1.2050);
    \coordinate (g3v8) at (1.2050,-0.9110);
    \coordinate (g3v9) at (0.7290,-0.1920);
    \draw[obstruction edge] (g3v1) -- (g3v2);
    \draw[obstruction edge] (g3v2) -- (g3v3);
    \draw[obstruction edge] (g3v1) -- (g3v5);
    \draw[obstruction edge] (g3v3) -- (g3v6);
    \draw[obstruction edge] (g3v4) -- (g3v6);
    \draw[obstruction edge] (g3v5) -- (g3v6);
    \draw[obstruction edge] (g3v1) -- (g3v7);
    \draw[obstruction edge] (g3v3) -- (g3v7);
    \draw[obstruction edge] (g3v4) -- (g3v8);
    \draw[obstruction edge] (g3v7) -- (g3v8);
    \draw[obstruction edge] (g3v3) -- (g3v9);
    \draw[obstruction edge] (g3v8) -- (g3v9);
    \node[obstruction vertex] at (g3v1) {};
    \node[obstruction vertex] at (g3v2) {};
    \node[obstruction vertex] at (g3v3) {};
    \node[obstruction vertex] at (g3v4) {};
    \node[obstruction vertex] at (g3v5) {};
    \node[obstruction vertex] at (g3v6) {};
    \node[obstruction vertex] at (g3v7) {};
    \node[obstruction vertex] at (g3v8) {};
    \node[obstruction vertex] at (g3v9) {};
  \end{tikzpicture}}

\expandafter\def\csname obstructiongraph@4\endcsname{\begin{tikzpicture}[x=1cm,y=1cm,baseline=(current bounding box.center)]
    \coordinate (g4v1) at (-1.0780,0.9243);
    \coordinate (g4v2) at (1.0780,0.9243);
    \coordinate (g4v3) at (-0.7607,-0.8460);
    \coordinate (g4v4) at (0.0000,-0.7631);
    \coordinate (g4v5) at (0.0000,-0.0743);
    \coordinate (g4v6) at (0.7607,-0.8460);
    \coordinate (g4v7) at (1.0823,-0.1658);
    \coordinate (g4v8) at (-1.0823,-0.1658);
    \coordinate (g4v9) at (0.0000,1.0123);
    \draw[obstruction edge] (g4v3) -- (g4v4);
    \draw[obstruction edge] (g4v3) -- (g4v5);
    \draw[obstruction edge] (g4v4) -- (g4v6);
    \draw[obstruction edge] (g4v5) -- (g4v6);
    \draw[obstruction edge] (g4v2) -- (g4v7);
    \draw[obstruction edge] (g4v4) -- (g4v7);
    \draw[obstruction edge] (g4v6) -- (g4v7);
    \draw[obstruction edge] (g4v1) -- (g4v8);
    \draw[obstruction edge] (g4v3) -- (g4v8);
    \draw[obstruction edge] (g4v4) -- (g4v8);
    \draw[obstruction edge] (g4v1) -- (g4v9);
    \draw[obstruction edge] (g4v2) -- (g4v9);
    \draw[obstruction edge] (g4v5) -- (g4v9);
    \node[obstruction vertex] at (g4v1) {};
    \node[obstruction vertex] at (g4v2) {};
    \node[obstruction vertex] at (g4v3) {};
    \node[obstruction vertex] at (g4v4) {};
    \node[obstruction vertex] at (g4v5) {};
    \node[obstruction vertex] at (g4v6) {};
    \node[obstruction vertex] at (g4v7) {};
    \node[obstruction vertex] at (g4v8) {};
    \node[obstruction vertex] at (g4v9) {};
  \end{tikzpicture}}

\expandafter\def\csname obstructiongraph@5\endcsname{\begin{tikzpicture}[x=1cm,y=1cm,baseline=(current bounding box.center)]
    \coordinate (g5v1) at (0.0000,1.4200);
    \coordinate (g5v2) at (-0.5588,0.6597);
    \coordinate (g5v3) at (1.3505,0.4388);
    \coordinate (g5v4) at (-0.8347,-1.1488);
    \coordinate (g5v5) at (-1.3505,0.4388);
    \coordinate (g5v6) at (0.8347,-1.1488);
    \coordinate (g5v7) at (0.2200,0.4800);
    \coordinate (g5v8) at (-0.3400,0.1000);
    \coordinate (g5v9) at (-0.0800,-0.4500);
    \draw[obstruction edge] (g5v1) -- (g5v2);
    \draw[obstruction edge] (g5v1) -- (g5v3);
    \draw[obstruction edge] (g5v1) -- (g5v5);
    \draw[obstruction edge] (g5v2) -- (g5v5);
    \draw[obstruction edge] (g5v4) -- (g5v5);
    \draw[obstruction edge] (g5v3) -- (g5v6);
    \draw[obstruction edge] (g5v4) -- (g5v6);
    \draw[obstruction edge] (g5v1) -- (g5v7);
    \draw[obstruction edge] (g5v2) -- (g5v8);
    \draw[obstruction edge] (g5v4) -- (g5v9);
    \draw[obstruction edge] (g5v6) -- (g5v9);
    \draw[obstruction edge] (g5v7) -- (g5v9);
    \draw[obstruction edge] (g5v8) -- (g5v9);
    \node[obstruction vertex] at (g5v1) {};
    \node[obstruction vertex] at (g5v2) {};
    \node[obstruction vertex] at (g5v3) {};
    \node[obstruction vertex] at (g5v4) {};
    \node[obstruction vertex] at (g5v5) {};
    \node[obstruction vertex] at (g5v6) {};
    \node[obstruction vertex] at (g5v7) {};
    \node[obstruction vertex] at (g5v8) {};
    \node[obstruction vertex] at (g5v9) {};
  \end{tikzpicture}}

\expandafter\def\csname obstructiongraph@6\endcsname{\begin{tikzpicture}[x=1cm,y=1cm,baseline=(current bounding box.center)]
    \coordinate (g6v1) at (-0.2682,0.3520);
    \coordinate (g6v2) at (-0.1960,1.1104);
    \coordinate (g6v3) at (0.0170,-0.4117);
    \coordinate (g6v4) at (0.6947,0.7570);
    \coordinate (g6v5) at (0.0703,-0.7978);
    \coordinate (g6v6) at (0.3308,0.2799);
    \coordinate (g6v7) at (-0.7073,-0.0175);
    \coordinate (g6v8) at (-0.9491,-1.0562);
    \coordinate (g6v9) at (1.0078,-0.2160);
    \draw[obstruction edge] (g6v1) -- (g6v2);
    \draw[obstruction edge] (g6v1) -- (g6v3);
    \draw[obstruction edge] (g6v1) -- (g6v4);
    \draw[obstruction edge] (g6v2) -- (g6v4);
    \draw[obstruction edge] (g6v2) -- (g6v6);
    \draw[obstruction edge] (g6v3) -- (g6v6);
    \draw[obstruction edge] (g6v2) -- (g6v7);
    \draw[obstruction edge] (g6v5) -- (g6v7);
    \draw[obstruction edge] (g6v6) -- (g6v7);
    \draw[obstruction edge] (g6v1) -- (g6v8);
    \draw[obstruction edge] (g6v3) -- (g6v8);
    \draw[obstruction edge] (g6v5) -- (g6v8);
    \draw[obstruction edge] (g6v7) -- (g6v8);
    \draw[obstruction edge] (g6v3) -- (g6v9);
    \draw[obstruction edge] (g6v4) -- (g6v9);
    \draw[obstruction edge] (g6v5) -- (g6v9);
    \draw[obstruction edge] (g6v6) -- (g6v9);
    \node[obstruction vertex] at (g6v1) {};
    \node[obstruction vertex] at (g6v2) {};
    \node[obstruction vertex] at (g6v3) {};
    \node[obstruction vertex] at (g6v4) {};
    \node[obstruction vertex] at (g6v5) {};
    \node[obstruction vertex] at (g6v6) {};
    \node[obstruction vertex] at (g6v7) {};
    \node[obstruction vertex] at (g6v8) {};
    \node[obstruction vertex] at (g6v9) {};
  \end{tikzpicture}}

\expandafter\def\csname obstructiongraph@7\endcsname{\begin{tikzpicture}[x=1cm,y=1cm,baseline=(current bounding box.center)]
    \coordinate (g7v1) at (-1.3817,-0.3274);
    \coordinate (g7v2) at (0.4886,-0.1623);
    \coordinate (g7v3) at (0.5344,0.6805);
    \coordinate (g7v4) at (-0.4501,0.0962);
    \coordinate (g7v5) at (1.0749,-0.2481);
    \coordinate (g7v6) at (0.4287,0.2688);
    \coordinate (g7v7) at (-0.1222,-0.7305);
    \coordinate (g7v8) at (-0.0056,-0.2129);
    \coordinate (g7v9) at (-0.5667,0.6358);
    \draw[obstruction edge] (g7v3) -- (g7v4);
    \draw[obstruction edge] (g7v3) -- (g7v5);
    \draw[obstruction edge] (g7v4) -- (g7v6);
    \draw[obstruction edge] (g7v5) -- (g7v6);
    \draw[obstruction edge] (g7v1) -- (g7v7);
    \draw[obstruction edge] (g7v2) -- (g7v7);
    \draw[obstruction edge] (g7v4) -- (g7v7);
    \draw[obstruction edge] (g7v5) -- (g7v7);
    \draw[obstruction edge] (g7v2) -- (g7v8);
    \draw[obstruction edge] (g7v4) -- (g7v8);
    \draw[obstruction edge] (g7v6) -- (g7v8);
    \draw[obstruction edge] (g7v1) -- (g7v9);
    \draw[obstruction edge] (g7v3) -- (g7v9);
    \draw[obstruction edge] (g7v4) -- (g7v9);
    \node[obstruction vertex] at (g7v1) {};
    \node[obstruction vertex] at (g7v2) {};
    \node[obstruction vertex] at (g7v3) {};
    \node[obstruction vertex] at (g7v4) {};
    \node[obstruction vertex] at (g7v5) {};
    \node[obstruction vertex] at (g7v6) {};
    \node[obstruction vertex] at (g7v7) {};
    \node[obstruction vertex] at (g7v8) {};
    \node[obstruction vertex] at (g7v9) {};
  \end{tikzpicture}}

\expandafter\def\csname obstructiongraph@8\endcsname{\begin{tikzpicture}[x=1cm,y=1cm,baseline=(current bounding box.center)]
    \coordinate (g8v1) at (0.3736,0.4049);
    \coordinate (g8v2) at (-0.5020,-0.8916);
    \coordinate (g8v3) at (1.1292,-0.3340);
    \coordinate (g8v4) at (-0.5676,0.0297);
    \coordinate (g8v5) at (0.2886,-0.0991);
    \coordinate (g8v6) at (0.4413,-1.0356);
    \coordinate (g8v7) at (-0.3662,0.6069);
    \coordinate (g8v8) at (-1.0342,-0.0812);
    \coordinate (g8v9) at (0.2373,1.4000);
    \draw[obstruction edge] (g8v1) -- (g8v3);
    \draw[obstruction edge] (g8v1) -- (g8v4);
    \draw[obstruction edge] (g8v2) -- (g8v4);
    \draw[obstruction edge] (g8v3) -- (g8v5);
    \draw[obstruction edge] (g8v4) -- (g8v5);
    \draw[obstruction edge] (g8v2) -- (g8v6);
    \draw[obstruction edge] (g8v3) -- (g8v6);
    \draw[obstruction edge] (g8v5) -- (g8v6);
    \draw[obstruction edge] (g8v5) -- (g8v7);
    \draw[obstruction edge] (g8v2) -- (g8v8);
    \draw[obstruction edge] (g8v7) -- (g8v8);
    \draw[obstruction edge] (g8v1) -- (g8v9);
    \draw[obstruction edge] (g8v7) -- (g8v9);
    \node[obstruction vertex] at (g8v1) {};
    \node[obstruction vertex] at (g8v2) {};
    \node[obstruction vertex] at (g8v3) {};
    \node[obstruction vertex] at (g8v4) {};
    \node[obstruction vertex] at (g8v5) {};
    \node[obstruction vertex] at (g8v6) {};
    \node[obstruction vertex] at (g8v7) {};
    \node[obstruction vertex] at (g8v8) {};
    \node[obstruction vertex] at (g8v9) {};
  \end{tikzpicture}}

\expandafter\def\csname obstructiongraph@9\endcsname{\begin{tikzpicture}[x=1cm,y=1cm,baseline=(current bounding box.center)]
    \coordinate (g9v1) at (0.0000,1.4200);
    \coordinate (g9v2) at (0.8347,-1.1488);
    \coordinate (g9v3) at (-0.8347,-1.1488);
    \coordinate (g9v4) at (-1.3505,0.4388);
    \coordinate (g9v5) at (1.3505,0.4388);
    \coordinate (g9v6) at (0.4173,-0.8313);
    \coordinate (g9v7) at (0.0000,0.1213);
    \coordinate (g9v8) at (-0.4173,-0.8313);
    \coordinate (g9v9) at (0.0000,-0.5138);
    \draw[obstruction edge] (g9v2) -- (g9v3);
    \draw[obstruction edge] (g9v1) -- (g9v4);
    \draw[obstruction edge] (g9v3) -- (g9v4);
    \draw[obstruction edge] (g9v1) -- (g9v5);
    \draw[obstruction edge] (g9v2) -- (g9v5);
    \draw[obstruction edge] (g9v2) -- (g9v6);
    \draw[obstruction edge] (g9v4) -- (g9v7);
    \draw[obstruction edge] (g9v5) -- (g9v7);
    \draw[obstruction edge] (g9v3) -- (g9v8);
    \draw[obstruction edge] (g9v6) -- (g9v9);
    \draw[obstruction edge] (g9v7) -- (g9v9);
    \draw[obstruction edge] (g9v8) -- (g9v9);
    \node[obstruction vertex] at (g9v1) {};
    \node[obstruction vertex] at (g9v2) {};
    \node[obstruction vertex] at (g9v3) {};
    \node[obstruction vertex] at (g9v4) {};
    \node[obstruction vertex] at (g9v5) {};
    \node[obstruction vertex] at (g9v6) {};
    \node[obstruction vertex] at (g9v7) {};
    \node[obstruction vertex] at (g9v8) {};
    \node[obstruction vertex] at (g9v9) {};
  \end{tikzpicture}}

\expandafter\def\csname obstructiongraph@10\endcsname{\begin{tikzpicture}[x=1cm,y=1cm,baseline=(current bounding box.center)]
    \coordinate (g10v1) at (0.8361,-0.3600);
    \coordinate (g10v2) at (-0.8361,-0.3600);
    \coordinate (g10v3) at (1.4136,0.1343);
    \coordinate (g10v4) at (-1.4136,0.1343);
    \coordinate (g10v5) at (-0.5899,1.0285);
    \coordinate (g10v6) at (-0.6659,-0.9850);
    \coordinate (g10v7) at (0.5899,1.0285);
    \coordinate (g10v8) at (0.6659,-0.9850);
    \coordinate (g10v9) at (0.0000,0.3644);
    \draw[obstruction edge] (g10v4) -- (g10v5);
    \draw[obstruction edge] (g10v2) -- (g10v6);
    \draw[obstruction edge] (g10v4) -- (g10v6);
    \draw[obstruction edge] (g10v3) -- (g10v7);
    \draw[obstruction edge] (g10v5) -- (g10v7);
    \draw[obstruction edge] (g10v1) -- (g10v8);
    \draw[obstruction edge] (g10v3) -- (g10v8);
    \draw[obstruction edge] (g10v6) -- (g10v8);
    \draw[obstruction edge] (g10v1) -- (g10v9);
    \draw[obstruction edge] (g10v2) -- (g10v9);
    \draw[obstruction edge] (g10v5) -- (g10v9);
    \draw[obstruction edge] (g10v7) -- (g10v9);
    \node[obstruction vertex] at (g10v1) {};
    \node[obstruction vertex] at (g10v2) {};
    \node[obstruction vertex] at (g10v3) {};
    \node[obstruction vertex] at (g10v4) {};
    \node[obstruction vertex] at (g10v5) {};
    \node[obstruction vertex] at (g10v6) {};
    \node[obstruction vertex] at (g10v7) {};
    \node[obstruction vertex] at (g10v8) {};
    \node[obstruction vertex] at (g10v9) {};
  \end{tikzpicture}}

\expandafter\def\csname obstructiongraph@11\endcsname{\begin{tikzpicture}[x=1cm,y=1cm,baseline=(current bounding box.center)]
    \coordinate (g11v1) at (-0.7322,-0.7086);
    \coordinate (g11v2) at (0.7322,-0.7086);
    \coordinate (g11v3) at (-0.6453,-0.0730);
    \coordinate (g11v4) at (0.0000,-1.2632);
    \coordinate (g11v5) at (0.0000,0.1259);
    \coordinate (g11v6) at (0.0000,1.2433);
    \coordinate (g11v7) at (1.2188,0.7286);
    \coordinate (g11v8) at (0.6453,-0.0730);
    \coordinate (g11v9) at (-1.2188,0.7286);
    \draw[obstruction edge] (g11v1) -- (g11v2);
    \draw[obstruction edge] (g11v1) -- (g11v3);
    \draw[obstruction edge] (g11v1) -- (g11v4);
    \draw[obstruction edge] (g11v2) -- (g11v4);
    \draw[obstruction edge] (g11v3) -- (g11v5);
    \draw[obstruction edge] (g11v4) -- (g11v5);
    \draw[obstruction edge] (g11v5) -- (g11v6);
    \draw[obstruction edge] (g11v2) -- (g11v7);
    \draw[obstruction edge] (g11v6) -- (g11v7);
    \draw[obstruction edge] (g11v2) -- (g11v8);
    \draw[obstruction edge] (g11v3) -- (g11v8);
    \draw[obstruction edge] (g11v5) -- (g11v8);
    \draw[obstruction edge] (g11v1) -- (g11v9);
    \draw[obstruction edge] (g11v6) -- (g11v9);
    \node[obstruction vertex] at (g11v1) {};
    \node[obstruction vertex] at (g11v2) {};
    \node[obstruction vertex] at (g11v3) {};
    \node[obstruction vertex] at (g11v4) {};
    \node[obstruction vertex] at (g11v5) {};
    \node[obstruction vertex] at (g11v6) {};
    \node[obstruction vertex] at (g11v7) {};
    \node[obstruction vertex] at (g11v8) {};
    \node[obstruction vertex] at (g11v9) {};
  \end{tikzpicture}}

\expandafter\def\csname obstructiongraph@12\endcsname{\begin{tikzpicture}[x=1cm,y=1cm,baseline=(current bounding box.center)]
    \coordinate (g12v1) at (-0.5330,0.6190);
    \coordinate (g12v2) at (-0.4195,-0.8784);
    \coordinate (g12v3) at (0.2417,-0.2461);
    \coordinate (g12v4) at (-0.2417,-0.2461);
    \coordinate (g12v5) at (0.4195,-0.8784);
    \coordinate (g12v6) at (-1.0464,-0.2046);
    \coordinate (g12v7) at (0.5330,0.6190);
    \coordinate (g12v8) at (0.0000,1.4200);
    \coordinate (g12v9) at (1.0464,-0.2046);
    \draw[obstruction edge] (g12v1) -- (g12v3);
    \draw[obstruction edge] (g12v2) -- (g12v3);
    \draw[obstruction edge] (g12v2) -- (g12v4);
    \draw[obstruction edge] (g12v2) -- (g12v5);
    \draw[obstruction edge] (g12v3) -- (g12v5);
    \draw[obstruction edge] (g12v4) -- (g12v5);
    \draw[obstruction edge] (g12v1) -- (g12v6);
    \draw[obstruction edge] (g12v2) -- (g12v6);
    \draw[obstruction edge] (g12v4) -- (g12v6);
    \draw[obstruction edge] (g12v4) -- (g12v7);
    \draw[obstruction edge] (g12v1) -- (g12v8);
    \draw[obstruction edge] (g12v7) -- (g12v8);
    \draw[obstruction edge] (g12v3) -- (g12v9);
    \draw[obstruction edge] (g12v5) -- (g12v9);
    \draw[obstruction edge] (g12v7) -- (g12v9);
    \node[obstruction vertex] at (g12v1) {};
    \node[obstruction vertex] at (g12v2) {};
    \node[obstruction vertex] at (g12v3) {};
    \node[obstruction vertex] at (g12v4) {};
    \node[obstruction vertex] at (g12v5) {};
    \node[obstruction vertex] at (g12v6) {};
    \node[obstruction vertex] at (g12v7) {};
    \node[obstruction vertex] at (g12v8) {};
    \node[obstruction vertex] at (g12v9) {};
  \end{tikzpicture}}

\expandafter\def\csname obstructiongraph@13\endcsname{\begin{tikzpicture}[x=1cm,y=1cm,baseline=(current bounding box.center)]
    \coordinate (g13v1) at (1.0030,0.5140);
    \coordinate (g13v2) at (0.0000,0.7676);
    \coordinate (g13v3) at (-1.3771,-0.3465);
    \coordinate (g13v4) at (-0.4111,-0.6728);
    \coordinate (g13v5) at (0.3385,0.1214);
    \coordinate (g13v6) at (-1.0030,0.5140);
    \coordinate (g13v7) at (-0.3385,0.1214);
    \coordinate (g13v8) at (1.3771,-0.3465);
    \coordinate (g13v9) at (0.4111,-0.6728);
    \draw[obstruction edge] (g13v1) -- (g13v2);
    \draw[obstruction edge] (g13v3) -- (g13v4);
    \draw[obstruction edge] (g13v1) -- (g13v5);
    \draw[obstruction edge] (g13v2) -- (g13v5);
    \draw[obstruction edge] (g13v4) -- (g13v5);
    \draw[obstruction edge] (g13v2) -- (g13v6);
    \draw[obstruction edge] (g13v3) -- (g13v6);
    \draw[obstruction edge] (g13v2) -- (g13v7);
    \draw[obstruction edge] (g13v6) -- (g13v7);
    \draw[obstruction edge] (g13v1) -- (g13v8);
    \draw[obstruction edge] (g13v4) -- (g13v9);
    \draw[obstruction edge] (g13v7) -- (g13v9);
    \draw[obstruction edge] (g13v8) -- (g13v9);
    \node[obstruction vertex] at (g13v1) {};
    \node[obstruction vertex] at (g13v2) {};
    \node[obstruction vertex] at (g13v3) {};
    \node[obstruction vertex] at (g13v4) {};
    \node[obstruction vertex] at (g13v5) {};
    \node[obstruction vertex] at (g13v6) {};
    \node[obstruction vertex] at (g13v7) {};
    \node[obstruction vertex] at (g13v8) {};
    \node[obstruction vertex] at (g13v9) {};
  \end{tikzpicture}}

\expandafter\def\csname obstructiongraph@14\endcsname{\begin{tikzpicture}[x=1cm,y=1cm,baseline=(current bounding box.center)]
    \coordinate (g14v1) at (0.6919,1.0271);
    \coordinate (g14v2) at (-1.4184,-0.0667);
    \coordinate (g14v3) at (0.6291,-1.0694);
    \coordinate (g14v4) at (-0.6291,-1.0694);
    \coordinate (g14v5) at (1.4184,-0.0667);
    \coordinate (g14v6) at (-0.6919,1.0271);
    \coordinate (g14v7) at (-0.3967,0.3706);
    \coordinate (g14v8) at (0.3967,0.3706);
    \coordinate (g14v9) at (0.0000,-0.5235);
    \draw[obstruction edge] (g14v2) -- (g14v4);
    \draw[obstruction edge] (g14v3) -- (g14v4);
    \draw[obstruction edge] (g14v1) -- (g14v5);
    \draw[obstruction edge] (g14v3) -- (g14v5);
    \draw[obstruction edge] (g14v1) -- (g14v6);
    \draw[obstruction edge] (g14v2) -- (g14v6);
    \draw[obstruction edge] (g14v6) -- (g14v7);
    \draw[obstruction edge] (g14v1) -- (g14v8);
    \draw[obstruction edge] (g14v7) -- (g14v8);
    \draw[obstruction edge] (g14v3) -- (g14v9);
    \draw[obstruction edge] (g14v4) -- (g14v9);
    \draw[obstruction edge] (g14v7) -- (g14v9);
    \draw[obstruction edge] (g14v8) -- (g14v9);
    \node[obstruction vertex] at (g14v1) {};
    \node[obstruction vertex] at (g14v2) {};
    \node[obstruction vertex] at (g14v3) {};
    \node[obstruction vertex] at (g14v4) {};
    \node[obstruction vertex] at (g14v5) {};
    \node[obstruction vertex] at (g14v6) {};
    \node[obstruction vertex] at (g14v7) {};
    \node[obstruction vertex] at (g14v8) {};
    \node[obstruction vertex] at (g14v9) {};
  \end{tikzpicture}}

\expandafter\def\csname obstructiongraph@15\endcsname{\begin{tikzpicture}[x=1cm,y=1cm,baseline=(current bounding box.center)]
    \coordinate (g15v1) at (0.7016,0.6385);
    \coordinate (g15v2) at (0.9661,-1.0407);
    \coordinate (g15v3) at (-0.7016,0.6385);
    \coordinate (g15v4) at (-1.3772,0.1832);
    \coordinate (g15v5) at (0.0000,1.4067);
    \coordinate (g15v6) at (-0.3305,-0.4844);
    \coordinate (g15v7) at (0.3305,-0.4844);
    \coordinate (g15v8) at (-0.9661,-1.0407);
    \coordinate (g15v9) at (1.3772,0.1832);
    \draw[obstruction edge] (g15v1) -- (g15v2);
    \draw[obstruction edge] (g15v1) -- (g15v3);
    \draw[obstruction edge] (g15v1) -- (g15v5);
    \draw[obstruction edge] (g15v3) -- (g15v5);
    \draw[obstruction edge] (g15v4) -- (g15v5);
    \draw[obstruction edge] (g15v3) -- (g15v6);
    \draw[obstruction edge] (g15v1) -- (g15v7);
    \draw[obstruction edge] (g15v2) -- (g15v7);
    \draw[obstruction edge] (g15v4) -- (g15v7);
    \draw[obstruction edge] (g15v2) -- (g15v8);
    \draw[obstruction edge] (g15v3) -- (g15v8);
    \draw[obstruction edge] (g15v4) -- (g15v8);
    \draw[obstruction edge] (g15v6) -- (g15v8);
    \draw[obstruction edge] (g15v2) -- (g15v9);
    \draw[obstruction edge] (g15v5) -- (g15v9);
    \draw[obstruction edge] (g15v6) -- (g15v9);
    \node[obstruction vertex] at (g15v1) {};
    \node[obstruction vertex] at (g15v2) {};
    \node[obstruction vertex] at (g15v3) {};
    \node[obstruction vertex] at (g15v4) {};
    \node[obstruction vertex] at (g15v5) {};
    \node[obstruction vertex] at (g15v6) {};
    \node[obstruction vertex] at (g15v7) {};
    \node[obstruction vertex] at (g15v8) {};
    \node[obstruction vertex] at (g15v9) {};
  \end{tikzpicture}}

\expandafter\def\csname obstructiongraph@16\endcsname{\begin{tikzpicture}[x=1cm,y=1cm,baseline=(current bounding box.center)]
    \coordinate (g16v1) at (-0.2569,1.1358);
    \coordinate (g16v2) at (0.3218,-1.3831);
    \coordinate (g16v3) at (0.6440,0.0163);
    \coordinate (g16v4) at (0.2901,-0.2046);
    \coordinate (g16v5) at (-0.1346,0.0690);
    \coordinate (g16v6) at (-1.0572,-0.2819);
    \coordinate (g16v7) at (1.0365,0.1508);
    \coordinate (g16v8) at (-0.5436,0.0428);
    \coordinate (g16v9) at (-0.3001,0.4550);
    \draw[obstruction edge] (g16v2) -- (g16v3);
    \draw[obstruction edge] (g16v1) -- (g16v4);
    \draw[obstruction edge] (g16v3) -- (g16v4);
    \draw[obstruction edge] (g16v4) -- (g16v5);
    \draw[obstruction edge] (g16v1) -- (g16v6);
    \draw[obstruction edge] (g16v2) -- (g16v6);
    \draw[obstruction edge] (g16v4) -- (g16v6);
    \draw[obstruction edge] (g16v1) -- (g16v7);
    \draw[obstruction edge] (g16v2) -- (g16v7);
    \draw[obstruction edge] (g16v3) -- (g16v7);
    \draw[obstruction edge] (g16v5) -- (g16v8);
    \draw[obstruction edge] (g16v6) -- (g16v8);
    \draw[obstruction edge] (g16v1) -- (g16v9);
    \draw[obstruction edge] (g16v5) -- (g16v9);
    \draw[obstruction edge] (g16v8) -- (g16v9);
    \node[obstruction vertex] at (g16v1) {};
    \node[obstruction vertex] at (g16v2) {};
    \node[obstruction vertex] at (g16v3) {};
    \node[obstruction vertex] at (g16v4) {};
    \node[obstruction vertex] at (g16v5) {};
    \node[obstruction vertex] at (g16v6) {};
    \node[obstruction vertex] at (g16v7) {};
    \node[obstruction vertex] at (g16v8) {};
    \node[obstruction vertex] at (g16v9) {};
  \end{tikzpicture}}

\expandafter\def\csname obstructiongraph@17\endcsname{\begin{tikzpicture}[x=1cm,y=1cm,baseline=(current bounding box.center)]
    \coordinate (g17v1) at (0.9600,0.0000);
    \coordinate (g17v2) at (0.2466,1.3984);
    \coordinate (g17v3) at (1.0878,-0.9128);
    \coordinate (g17v4) at (-0.4800,0.8314);
    \coordinate (g17v5) at (0.3677,0.3085);
    \coordinate (g17v6) at (-0.4511,0.1642);
    \coordinate (g17v7) at (0.0834,-0.4727);
    \coordinate (g17v8) at (-1.3344,-0.4857);
    \coordinate (g17v9) at (-0.4800,-0.8314);
    \draw[obstruction edge] (g17v1) -- (g17v2);
    \draw[obstruction edge] (g17v1) -- (g17v3);
    \draw[obstruction edge] (g17v2) -- (g17v4);
    \draw[obstruction edge] (g17v3) -- (g17v4);
    \draw[obstruction edge] (g17v1) -- (g17v5);
    \draw[obstruction edge] (g17v2) -- (g17v5);
    \draw[obstruction edge] (g17v4) -- (g17v6);
    \draw[obstruction edge] (g17v5) -- (g17v6);
    \draw[obstruction edge] (g17v3) -- (g17v7);
    \draw[obstruction edge] (g17v5) -- (g17v7);
    \draw[obstruction edge] (g17v6) -- (g17v7);
    \draw[obstruction edge] (g17v1) -- (g17v8);
    \draw[obstruction edge] (g17v4) -- (g17v8);
    \draw[obstruction edge] (g17v6) -- (g17v8);
    \draw[obstruction edge] (g17v2) -- (g17v9);
    \draw[obstruction edge] (g17v3) -- (g17v9);
    \draw[obstruction edge] (g17v7) -- (g17v9);
    \draw[obstruction edge] (g17v8) -- (g17v9);
    \node[obstruction vertex] at (g17v1) {};
    \node[obstruction vertex] at (g17v2) {};
    \node[obstruction vertex] at (g17v3) {};
    \node[obstruction vertex] at (g17v4) {};
    \node[obstruction vertex] at (g17v5) {};
    \node[obstruction vertex] at (g17v6) {};
    \node[obstruction vertex] at (g17v7) {};
    \node[obstruction vertex] at (g17v8) {};
    \node[obstruction vertex] at (g17v9) {};
  \end{tikzpicture}}

\expandafter\def\csname obstructiongraph@18\endcsname{\begin{tikzpicture}[x=1cm,y=1cm,baseline=(current bounding box.center)]
    \coordinate (g18v1) at (-0.5749,-0.1141);
    \coordinate (g18v2) at (0.8146,-1.1457);
    \coordinate (g18v3) at (-1.3928,-0.2767);
    \coordinate (g18v4) at (0.1692,-0.3872);
    \coordinate (g18v5) at (0.9720,0.1911);
    \coordinate (g18v6) at (-0.3239,0.8460);
    \coordinate (g18v7) at (0.0030,0.4212);
    \coordinate (g18v8) at (0.2980,1.3634);
    \coordinate (g18v9) at (0.0348,-0.8980);
    \draw[obstruction edge] (g18v1) -- (g18v2);
    \draw[obstruction edge] (g18v1) -- (g18v3);
    \draw[obstruction edge] (g18v2) -- (g18v4);
    \draw[obstruction edge] (g18v3) -- (g18v4);
    \draw[obstruction edge] (g18v2) -- (g18v5);
    \draw[obstruction edge] (g18v4) -- (g18v5);
    \draw[obstruction edge] (g18v1) -- (g18v6);
    \draw[obstruction edge] (g18v3) -- (g18v6);
    \draw[obstruction edge] (g18v5) -- (g18v6);
    \draw[obstruction edge] (g18v2) -- (g18v7);
    \draw[obstruction edge] (g18v3) -- (g18v7);
    \draw[obstruction edge] (g18v5) -- (g18v7);
    \draw[obstruction edge] (g18v6) -- (g18v7);
    \draw[obstruction edge] (g18v1) -- (g18v8);
    \draw[obstruction edge] (g18v4) -- (g18v8);
    \draw[obstruction edge] (g18v5) -- (g18v8);
    \draw[obstruction edge] (g18v6) -- (g18v8);
    \draw[obstruction edge] (g18v7) -- (g18v8);
    \draw[obstruction edge] (g18v1) -- (g18v9);
    \draw[obstruction edge] (g18v2) -- (g18v9);
    \draw[obstruction edge] (g18v3) -- (g18v9);
    \draw[obstruction edge] (g18v4) -- (g18v9);
    \draw[obstruction edge] (g18v5) -- (g18v9);
    \draw[obstruction edge] (g18v6) -- (g18v9);
    \node[obstruction vertex] at (g18v1) {};
    \node[obstruction vertex] at (g18v2) {};
    \node[obstruction vertex] at (g18v3) {};
    \node[obstruction vertex] at (g18v4) {};
    \node[obstruction vertex] at (g18v5) {};
    \node[obstruction vertex] at (g18v6) {};
    \node[obstruction vertex] at (g18v7) {};
    \node[obstruction vertex] at (g18v8) {};
    \node[obstruction vertex] at (g18v9) {};
  \end{tikzpicture}}

\expandafter\def\csname obstructiongraph@19\endcsname{\begin{tikzpicture}[x=1cm,y=1cm,baseline=(current bounding box.center)]
    \coordinate (g19v1) at (1.4200,0.0000);
    \coordinate (g19v2) at (1.0878,0.9128);
    \coordinate (g19v3) at (-1.3344,-0.4857);
    \coordinate (g19v4) at (0.2466,1.3984);
    \coordinate (g19v5) at (1.0878,-0.9128);
    \coordinate (g19v6) at (0.2466,-1.3984);
    \coordinate (g19v7) at (-0.7100,-1.2298);
    \coordinate (g19v8) at (-1.3344,0.4857);
    \coordinate (g19v9) at (-0.7100,1.2298);
    \draw[obstruction edge] (g19v1) -- (g19v2);
    \draw[obstruction edge] (g19v2) -- (g19v4);
    \draw[obstruction edge] (g19v3) -- (g19v4);
    \draw[obstruction edge] (g19v1) -- (g19v5);
    \draw[obstruction edge] (g19v3) -- (g19v5);
    \draw[obstruction edge] (g19v4) -- (g19v5);
    \draw[obstruction edge] (g19v2) -- (g19v6);
    \draw[obstruction edge] (g19v5) -- (g19v6);
    \draw[obstruction edge] (g19v1) -- (g19v7);
    \draw[obstruction edge] (g19v3) -- (g19v7);
    \draw[obstruction edge] (g19v6) -- (g19v7);
    \draw[obstruction edge] (g19v2) -- (g19v8);
    \draw[obstruction edge] (g19v3) -- (g19v8);
    \draw[obstruction edge] (g19v6) -- (g19v8);
    \draw[obstruction edge] (g19v1) -- (g19v9);
    \draw[obstruction edge] (g19v4) -- (g19v9);
    \draw[obstruction edge] (g19v7) -- (g19v9);
    \draw[obstruction edge] (g19v8) -- (g19v9);
    \node[obstruction vertex] at (g19v1) {};
    \node[obstruction vertex] at (g19v2) {};
    \node[obstruction vertex] at (g19v3) {};
    \node[obstruction vertex] at (g19v4) {};
    \node[obstruction vertex] at (g19v5) {};
    \node[obstruction vertex] at (g19v6) {};
    \node[obstruction vertex] at (g19v7) {};
    \node[obstruction vertex] at (g19v8) {};
    \node[obstruction vertex] at (g19v9) {};
  \end{tikzpicture}}

\expandafter\def\csname obstructiongraph@20\endcsname{\begin{tikzpicture}[x=1cm,y=1cm,baseline=(current bounding box.center)]
    \coordinate (g20v1) at (-0.2170,-0.6975);
    \coordinate (g20v2) at (0.5963,-0.4219);
    \coordinate (g20v3) at (-0.7304,-0.0091);
    \coordinate (g20v4) at (1.1591,-0.8202);
    \coordinate (g20v5) at (-0.4219,-1.3559);
    \coordinate (g20v6) at (1.1383,0.8489);
    \coordinate (g20v7) at (0.5855,0.4367);
    \coordinate (g20v8) at (-0.4556,1.3449);
    \coordinate (g20v9) at (-1.4199,-0.0177);
    \coordinate (g20v10) at (-0.2344,0.6918);
    \draw[obstruction edge] (g20v1) -- (g20v2);
    \draw[obstruction edge] (g20v1) -- (g20v3);
    \draw[obstruction edge] (g20v1) -- (g20v4);
    \draw[obstruction edge] (g20v2) -- (g20v5);
    \draw[obstruction edge] (g20v3) -- (g20v5);
    \draw[obstruction edge] (g20v2) -- (g20v6);
    \draw[obstruction edge] (g20v2) -- (g20v7);
    \draw[obstruction edge] (g20v4) -- (g20v7);
    \draw[obstruction edge] (g20v3) -- (g20v8);
    \draw[obstruction edge] (g20v7) -- (g20v8);
    \draw[obstruction edge] (g20v1) -- (g20v9);
    \draw[obstruction edge] (g20v3) -- (g20v10);
    \draw[obstruction edge] (g20v6) -- (g20v10);
    \draw[obstruction edge] (g20v7) -- (g20v10);
    \draw[obstruction edge] (g20v9) -- (g20v10);
    \node[obstruction vertex] at (g20v1) {};
    \node[obstruction vertex] at (g20v2) {};
    \node[obstruction vertex] at (g20v3) {};
    \node[obstruction vertex] at (g20v4) {};
    \node[obstruction vertex] at (g20v5) {};
    \node[obstruction vertex] at (g20v6) {};
    \node[obstruction vertex] at (g20v7) {};
    \node[obstruction vertex] at (g20v8) {};
    \node[obstruction vertex] at (g20v9) {};
    \node[obstruction vertex] at (g20v10) {};
  \end{tikzpicture}}

\expandafter\def\csname obstructiongraph@21\endcsname{\begin{tikzpicture}[x=1cm,y=1cm,baseline=(current bounding box.center)]
    \coordinate (g21v1) at (1.4200,0.0000);
    \coordinate (g21v2) at (0.4388,1.3505);
    \coordinate (g21v3) at (0.6700,0.0000);
    \coordinate (g21v4) at (-0.5420,-0.3938);
    \coordinate (g21v5) at (-1.1488,-0.8347);
    \coordinate (g21v6) at (0.4388,-1.3505);
    \coordinate (g21v7) at (-1.1488,0.8347);
    \coordinate (g21v8) at (-0.5420,0.3938);
    \coordinate (g21v9) at (0.2070,0.6372);
    \coordinate (g21v10) at (0.2070,-0.6372);
    \draw[obstruction edge] (g21v1) -- (g21v2);
    \draw[obstruction edge] (g21v3) -- (g21v4);
    \draw[obstruction edge] (g21v1) -- (g21v6);
    \draw[obstruction edge] (g21v3) -- (g21v6);
    \draw[obstruction edge] (g21v5) -- (g21v6);
    \draw[obstruction edge] (g21v2) -- (g21v7);
    \draw[obstruction edge] (g21v4) -- (g21v7);
    \draw[obstruction edge] (g21v5) -- (g21v7);
    \draw[obstruction edge] (g21v2) -- (g21v8);
    \draw[obstruction edge] (g21v3) -- (g21v8);
    \draw[obstruction edge] (g21v1) -- (g21v9);
    \draw[obstruction edge] (g21v4) -- (g21v9);
    \draw[obstruction edge] (g21v5) -- (g21v10);
    \draw[obstruction edge] (g21v8) -- (g21v10);
    \draw[obstruction edge] (g21v9) -- (g21v10);
    \node[obstruction vertex] at (g21v1) {};
    \node[obstruction vertex] at (g21v2) {};
    \node[obstruction vertex] at (g21v3) {};
    \node[obstruction vertex] at (g21v4) {};
    \node[obstruction vertex] at (g21v5) {};
    \node[obstruction vertex] at (g21v6) {};
    \node[obstruction vertex] at (g21v7) {};
    \node[obstruction vertex] at (g21v8) {};
    \node[obstruction vertex] at (g21v9) {};
    \node[obstruction vertex] at (g21v10) {};
  \end{tikzpicture}}

\expandafter\def\csname obstructiongraph@22\endcsname{\begin{tikzpicture}[x=1cm,y=1cm,baseline=(current bounding box.center)]
    \coordinate (g22v1) at (0.9395,0.1532);
    \coordinate (g22v2) at (1.1883,-0.7759);
    \coordinate (g22v3) at (-1.2660,-0.6411);
    \coordinate (g22v4) at (-0.0000,-0.0000);
    \coordinate (g22v5) at (-1.4039,0.2135);
    \coordinate (g22v6) at (-0.6024,0.7371);
    \coordinate (g22v7) at (0.8868,1.1090);
    \coordinate (g22v8) at (-0.3371,-0.8903);
    \coordinate (g22v9) at (0.0778,1.4170);
    \coordinate (g22v10) at (0.5170,-1.3225);
    \draw[obstruction edge] (g22v1) -- (g22v2);
    \draw[obstruction edge] (g22v1) -- (g22v4);
    \draw[obstruction edge] (g22v3) -- (g22v5);
    \draw[obstruction edge] (g22v4) -- (g22v6);
    \draw[obstruction edge] (g22v5) -- (g22v6);
    \draw[obstruction edge] (g22v1) -- (g22v7);
    \draw[obstruction edge] (g22v3) -- (g22v8);
    \draw[obstruction edge] (g22v4) -- (g22v8);
    \draw[obstruction edge] (g22v6) -- (g22v9);
    \draw[obstruction edge] (g22v7) -- (g22v9);
    \draw[obstruction edge] (g22v2) -- (g22v10);
    \draw[obstruction edge] (g22v8) -- (g22v10);
    \node[obstruction vertex] at (g22v1) {};
    \node[obstruction vertex] at (g22v2) {};
    \node[obstruction vertex] at (g22v3) {};
    \node[obstruction vertex] at (g22v4) {};
    \node[obstruction vertex] at (g22v5) {};
    \node[obstruction vertex] at (g22v6) {};
    \node[obstruction vertex] at (g22v7) {};
    \node[obstruction vertex] at (g22v8) {};
    \node[obstruction vertex] at (g22v9) {};
    \node[obstruction vertex] at (g22v10) {};
  \end{tikzpicture}}

\expandafter\def\csname obstructiongraph@23\endcsname{\begin{tikzpicture}[x=1cm,y=1cm,baseline=(current bounding box.center)]
    \coordinate (g23v1) at (0.4800,0.0000);
    \coordinate (g23v2) at (1.0878,0.9128);
    \coordinate (g23v3) at (-0.2400,0.4157);
    \coordinate (g23v4) at (0.1667,0.9454);
    \coordinate (g23v5) at (-0.2400,-0.4157);
    \coordinate (g23v6) at (-0.9021,-0.3283);
    \coordinate (g23v7) at (0.2466,-1.3984);
    \coordinate (g23v8) at (0.7354,-0.6171);
    \coordinate (g23v9) at (-1.3344,0.4857);
    \coordinate (g23v10) at (0.0000,0.0000);
    \draw[obstruction edge] (g23v1) -- (g23v2);
    \draw[obstruction edge] (g23v1) -- (g23v3);
    \draw[obstruction edge] (g23v2) -- (g23v3);
    \draw[obstruction edge] (g23v1) -- (g23v4);
    \draw[obstruction edge] (g23v2) -- (g23v4);
    \draw[obstruction edge] (g23v1) -- (g23v5);
    \draw[obstruction edge] (g23v3) -- (g23v5);
    \draw[obstruction edge] (g23v3) -- (g23v6);
    \draw[obstruction edge] (g23v1) -- (g23v7);
    \draw[obstruction edge] (g23v5) -- (g23v7);
    \draw[obstruction edge] (g23v6) -- (g23v7);
    \draw[obstruction edge] (g23v2) -- (g23v8);
    \draw[obstruction edge] (g23v5) -- (g23v8);
    \draw[obstruction edge] (g23v7) -- (g23v8);
    \draw[obstruction edge] (g23v3) -- (g23v9);
    \draw[obstruction edge] (g23v4) -- (g23v9);
    \draw[obstruction edge] (g23v5) -- (g23v9);
    \draw[obstruction edge] (g23v6) -- (g23v9);
    \draw[obstruction edge] (g23v2) -- (g23v10);
    \draw[obstruction edge] (g23v7) -- (g23v10);
    \draw[obstruction edge] (g23v9) -- (g23v10);
    \node[obstruction vertex] at (g23v1) {};
    \node[obstruction vertex] at (g23v2) {};
    \node[obstruction vertex] at (g23v3) {};
    \node[obstruction vertex] at (g23v4) {};
    \node[obstruction vertex] at (g23v5) {};
    \node[obstruction vertex] at (g23v6) {};
    \node[obstruction vertex] at (g23v7) {};
    \node[obstruction vertex] at (g23v8) {};
    \node[obstruction vertex] at (g23v9) {};
    \node[obstruction vertex] at (g23v10) {};
  \end{tikzpicture}}

\expandafter\def\csname obstructiongraph@24\endcsname{\begin{tikzpicture}[x=1cm,y=1cm,baseline=(current bounding box.center)]
    \coordinate (g24v1) at (0.4298,0.7110);
    \coordinate (g24v2) at (1.3233,-0.5151);
    \coordinate (g24v3) at (-0.4743,0.4442);
    \coordinate (g24v4) at (-0.7384,-0.2035);
    \coordinate (g24v5) at (-0.7688,0.7148);
    \coordinate (g24v6) at (0.3338,-0.3412);
    \coordinate (g24v7) at (-1.0809,-0.4110);
    \coordinate (g24v8) at (0.2231,0.3693);
    \coordinate (g24v9) at (0.0624,-0.6386);
    \coordinate (g24v10) at (0.6901,-0.1299);
    \draw[obstruction edge] (g24v1) -- (g24v2);
    \draw[obstruction edge] (g24v1) -- (g24v3);
    \draw[obstruction edge] (g24v3) -- (g24v4);
    \draw[obstruction edge] (g24v1) -- (g24v5);
    \draw[obstruction edge] (g24v4) -- (g24v5);
    \draw[obstruction edge] (g24v2) -- (g24v6);
    \draw[obstruction edge] (g24v5) -- (g24v7);
    \draw[obstruction edge] (g24v3) -- (g24v8);
    \draw[obstruction edge] (g24v6) -- (g24v8);
    \draw[obstruction edge] (g24v2) -- (g24v9);
    \draw[obstruction edge] (g24v3) -- (g24v9);
    \draw[obstruction edge] (g24v4) -- (g24v9);
    \draw[obstruction edge] (g24v7) -- (g24v9);
    \draw[obstruction edge] (g24v1) -- (g24v10);
    \draw[obstruction edge] (g24v2) -- (g24v10);
    \draw[obstruction edge] (g24v8) -- (g24v10);
    \node[obstruction vertex] at (g24v1) {};
    \node[obstruction vertex] at (g24v2) {};
    \node[obstruction vertex] at (g24v3) {};
    \node[obstruction vertex] at (g24v4) {};
    \node[obstruction vertex] at (g24v5) {};
    \node[obstruction vertex] at (g24v6) {};
    \node[obstruction vertex] at (g24v7) {};
    \node[obstruction vertex] at (g24v8) {};
    \node[obstruction vertex] at (g24v9) {};
    \node[obstruction vertex] at (g24v10) {};
  \end{tikzpicture}}

\expandafter\def\csname obstructiongraph@25\endcsname{\begin{tikzpicture}[x=1cm,y=1cm,baseline=(current bounding box.center)]
    \coordinate (g25v1) at (1.2779,0.0841);
    \coordinate (g25v2) at (0.6116,0.6317);
    \coordinate (g25v3) at (-0.0426,-0.1907);
    \coordinate (g25v4) at (0.1213,0.3172);
    \coordinate (g25v5) at (0.2666,-1.3360);
    \coordinate (g25v6) at (-0.8766,-0.5862);
    \coordinate (g25v7) at (0.1717,-0.7811);
    \coordinate (g25v8) at (-0.6648,0.3181);
    \coordinate (g25v9) at (0.4849,1.1028);
    \coordinate (g25v10) at (-1.3501,0.4402);
    \draw[obstruction edge] (g25v1) -- (g25v2);
    \draw[obstruction edge] (g25v1) -- (g25v3);
    \draw[obstruction edge] (g25v2) -- (g25v3);
    \draw[obstruction edge] (g25v1) -- (g25v4);
    \draw[obstruction edge] (g25v2) -- (g25v4);
    \draw[obstruction edge] (g25v3) -- (g25v4);
    \draw[obstruction edge] (g25v1) -- (g25v5);
    \draw[obstruction edge] (g25v2) -- (g25v5);
    \draw[obstruction edge] (g25v3) -- (g25v6);
    \draw[obstruction edge] (g25v4) -- (g25v6);
    \draw[obstruction edge] (g25v5) -- (g25v6);
    \draw[obstruction edge] (g25v2) -- (g25v7);
    \draw[obstruction edge] (g25v3) -- (g25v7);
    \draw[obstruction edge] (g25v4) -- (g25v7);
    \draw[obstruction edge] (g25v5) -- (g25v7);
    \draw[obstruction edge] (g25v6) -- (g25v7);
    \draw[obstruction edge] (g25v2) -- (g25v8);
    \draw[obstruction edge] (g25v3) -- (g25v8);
    \draw[obstruction edge] (g25v4) -- (g25v8);
    \draw[obstruction edge] (g25v5) -- (g25v8);
    \draw[obstruction edge] (g25v7) -- (g25v8);
    \draw[obstruction edge] (g25v1) -- (g25v9);
    \draw[obstruction edge] (g25v4) -- (g25v9);
    \draw[obstruction edge] (g25v7) -- (g25v9);
    \draw[obstruction edge] (g25v8) -- (g25v9);
    \draw[obstruction edge] (g25v2) -- (g25v10);
    \draw[obstruction edge] (g25v3) -- (g25v10);
    \draw[obstruction edge] (g25v6) -- (g25v10);
    \draw[obstruction edge] (g25v8) -- (g25v10);
    \draw[obstruction edge] (g25v9) -- (g25v10);
    \node[obstruction vertex] at (g25v1) {};
    \node[obstruction vertex] at (g25v2) {};
    \node[obstruction vertex] at (g25v3) {};
    \node[obstruction vertex] at (g25v4) {};
    \node[obstruction vertex] at (g25v5) {};
    \node[obstruction vertex] at (g25v6) {};
    \node[obstruction vertex] at (g25v7) {};
    \node[obstruction vertex] at (g25v8) {};
    \node[obstruction vertex] at (g25v9) {};
    \node[obstruction vertex] at (g25v10) {};
  \end{tikzpicture}}
 \newcommand{\prettyobstructionrows}{$O_{8,1}$ & 8 & 14 & \texttt{GBr\{HK} & $O_{9,1}$ & 9 & 24 & \texttt{H\textbar{}\textasciicircum{}\textbackslash{}bVF}\\
$O_{9,2}$ & 9 & 12 & \texttt{Hg\_\textbar{}?c\textasciigrave{}} & $O_{9,3}$ & 9 & 13 & \texttt{H@GYlbG}\\
$O_{9,4}$ & 9 & 13 & \texttt{HossA?V} & $O_{9,5}$ & 9 & 17 & \texttt{Hu@a\textbackslash{}S\{}\\
$O_{9,6}$ & 9 & 14 & \texttt{H@G]qio} & $O_{9,7}$ & 9 & 13 & \texttt{HULgQEA}\\
$O_{9,8}$ & 9 & 12 & \texttt{HLp?p?F} & $O_{9,9}$ & 9 & 12 & \texttt{H?DPTJI}\\
$O_{9,10}$ & 9 & 14 & \texttt{HuKIJQC} & $O_{9,11}$ & 9 & 15 & \texttt{HY\textasciicircum{}OcCi}\\
$O_{9,12}$ & 9 & 13 & \texttt{H\textasciigrave{}taK?R} & $O_{9,13}$ & 9 & 13 & \texttt{HBj?KCr}\\
$O_{9,14}$ & 9 & 16 & \texttt{HokebhK} & $O_{9,15}$ & 9 & 15 & \texttt{HLFV?YH}\\
$O_{9,16}$ & 9 & 18 & \texttt{HroX[hb} & $O_{9,17}$ & 9 & 24 & \texttt{HrUj[\textasciitilde{}\{}\\
$O_{9,18}$ & 9 & 18 & \texttt{HblLJIR} & $O_{10,1}$ & 10 & 15 & \texttt{IsXA\textasciigrave{}E?Hg}\\
$O_{10,2}$ & 10 & 15 & \texttt{I\textasciigrave{}AirAOAW} & $O_{10,3}$ & 10 & 12 & \texttt{IcG[@\_EOO}\\
$O_{10,4}$ & 10 & 21 & \texttt{I\}gcYS\{Og} & $O_{10,5}$ & 10 & 16 & \texttt{Ipd?PHqoO}\\
$O_{10,6}$ & 10 & 30 & \texttt{I\textasciitilde{}ozzuRXW} &  &  &  & \\
}
 \newcommand{\pivotobstructionrows}{\multicolumn{10}{@{}l}{\bfseries 8 vertices (2 records)}\\[-1pt]
\texttt{GDYz\textasciitilde{}o} & \texttt{Gq?g\textasciitilde{}o} &  &  &  &  &  &  &  & \\
\addlinespace[2pt]
\multicolumn{10}{@{}l}{\bfseries 9 vertices (447 records)}\\[-1pt]
\texttt{H??ZTjM} & \texttt{H??ZTjY} & \texttt{H??ZfQ\textbar{}} & \texttt{H??\textbackslash{}JrN} & \texttt{H??\textbackslash{}RbL} & \texttt{H??\textbackslash{}b\textasciicircum{}Y} & \texttt{H??\textbackslash{}j\textasciicircum{}[} & \texttt{H??\textbackslash{}rZ\textbackslash{}} & \texttt{H??]Rel} & \texttt{H??isjx}\\
\texttt{H??i\textasciitilde{}q\}} & \texttt{H??kzjW} & \texttt{H??y\textbar{}zi} & \texttt{H?@[Pfh} & \texttt{H?CZL\textasciicircum{}y} & \texttt{H?CZLrM} & \texttt{H?CZNQu} & \texttt{H?CZVIu} & \texttt{H?CZ\textasciicircum{}i\}} & \texttt{H?CitbD}\\
\texttt{H?CitzU} & \texttt{H?Ci\textasciitilde{}a\textbar{}} & \texttt{H?Ci\textasciitilde{}rt} & \texttt{H?CjCnY} & \texttt{H?CjeYV} & \texttt{H?Cjlz]} & \texttt{H?Cjmy\}} & \texttt{H?CjszF} & \texttt{H?CkjNX} & \texttt{H?CkrNS}\\
\texttt{H?ClAvT} & \texttt{H?Clanh} & \texttt{H?CmfH\}} & \texttt{H?CqTNy} & \texttt{H?Cq\textbackslash{}\textasciicircum{}u} & \texttt{H?CrUYu} & \texttt{H?CrUY\}} & \texttt{H?CrUZu} & \texttt{H?Cr\textbackslash{}rK} & \texttt{H?CruY\textbar{}}\\
\texttt{H?CsZ\textasciicircum{}w} & \texttt{H?CtQ\textasciicircum{}\textbar{}} & \texttt{H?CuJQz} & \texttt{H?CuJU\{} & \texttt{H?CuJU\textasciitilde{}} & \texttt{H?CurY\textbar{}} & \texttt{H?CvRY\textbackslash{}} & \texttt{H?CyTfa} & \texttt{H?Cy\textbackslash{}Vo} & \texttt{H?CzKvg}\\
\texttt{H?CzKvj} & \texttt{H?CzLrI} & \texttt{H?CzTfK} & \texttt{H?CzUny} & \texttt{H?CzVb]} & \texttt{H?Cz\textasciicircum{}bZ} & \texttt{H?CzdZI} & \texttt{H?CzeQn} & \texttt{H?CzeVn} & \texttt{H?CzuZf}\\
\texttt{H?Czuzm} & \texttt{H?CzvZu} & \texttt{H?C\{RFp} & \texttt{H?C\textbar{}AVp} & \texttt{H?C\textbar{}RfL} & \texttt{H?C\textbar{}Zrp} & \texttt{H?C\textbar{}bVL} & \texttt{H?C\textbar{}j\textasciicircum{}w} & \texttt{H?C\}Rf\{} & \texttt{H?C\}rym}\\
\texttt{H?DH\textbar{}\textasciicircum{}\{} & \texttt{H?DJDe\textbar{}} & \texttt{H?DK\textasciigrave{}\textasciicircum{}p} & \texttt{H?DLZqs} & \texttt{H?DM\textasciigrave{}\}n} & \texttt{H?DPVM]} & \texttt{H?DRLU\textbar{}} & \texttt{H?DTJQx} & \texttt{H?DTVLy} & \texttt{H?DZTvs}\\
\texttt{H?D\textbackslash{}Rny} & \texttt{H?D\textasciigrave{}]fw} & \texttt{H?D\textasciigrave{}szM} & \texttt{H?D\textasciigrave{}\}zm} & \texttt{H?Da\textbar{}Yw} & \texttt{H?Dhtf[} & \texttt{H?Dl\textasciigrave{}\textasciitilde{}Z} & \texttt{H?DleTr} & \texttt{H?DmPmx} & \texttt{H?DmP\textasciitilde{}r}\\
\texttt{H?Dnbq\textbar{}} & \texttt{H?DpUVq} & \texttt{H?Dp\textasciitilde{}Qz} & \texttt{H?DtPvL} & \texttt{H?DuTS\textasciitilde{}} & \texttt{H?EeZgz} & \texttt{H?EjeT\textbar{}} & \texttt{H?Emrh\textasciitilde{}} & \texttt{H?Emzhx} & \texttt{H?EprV\textbackslash{}}\\
\texttt{H?Er]o\textasciitilde{}} & \texttt{H?E\}Rdf} & \texttt{H?F@Zqu} & \texttt{H?FTZpf} & \texttt{H?FadTz} & \texttt{H?GYlrM} & \texttt{H?GZeM\textbar{}} & \texttt{H?GZeYv} & \texttt{H?G\textbackslash{}Qno} & \texttt{H?G]\textasciicircum{}h\}}\\
\texttt{H?Gycva} & \texttt{H?HG\textbar{}br} & \texttt{H?HIle\textbackslash{}} & \texttt{H?HO\{v\{} & \texttt{H?HPuUt} & \texttt{H?HVCp\textbar{}} & \texttt{H?HXeEx} & \texttt{H?H[\textasciigrave{}Vp} & \texttt{H?H\textbackslash{}e[\textasciitilde{}} & \texttt{H?HouEz}\\
\texttt{H?Hqtv\}} & \texttt{H?Hsqrb} & \texttt{H?IPq\textasciitilde{}\{} & \texttt{H?IpqrB} & \texttt{H?J\textbackslash{}rpv} & \texttt{H?J\_udz} & \texttt{H?KRMYv} & \texttt{H?Kci\textasciitilde{}k} & \texttt{H?K\textbar{}Qno} & \texttt{H?LIlfk}\\
\texttt{H?LJdn\textbackslash{}} & \texttt{H?LR\textbackslash{}zu} & \texttt{H?LSrNl} & \texttt{H?LTMtm} & \texttt{H?LT]zu} & \texttt{H?LT\textasciicircum{}j]} & \texttt{H?LTmXj} & \texttt{H?LTnRL} & \texttt{H?LUTnn} & \texttt{H?L[jVr}\\
\texttt{H?L[vNr} & \texttt{H?L\textbackslash{}nR\textasciitilde{}} & \texttt{H?L\textbackslash{}vHv} & \texttt{H?L\textbackslash{}vJV} & \texttt{H?L\textbackslash{}vJv} & \texttt{H?L\textbackslash{}vNs} & \texttt{H?L\textasciicircum{}Thv} & \texttt{H?L\_\}fl} & \texttt{H?LbszT} & \texttt{H?Lcirb}\\
\texttt{H?LciuN} & \texttt{H?Lcivn} & \texttt{H?Lcmrm} & \texttt{H?Lcmt\}} & \texttt{H?Lcun\{} & \texttt{H?LeQ\}v} & \texttt{H?LeSl\textbar{}} & \texttt{H?Lecx\}} & \texttt{H?Lecze} & \texttt{H?Leczm}\\
\texttt{H?Lecz\}} & \texttt{H?Lec\textasciitilde{}m} & \texttt{H?LrKvX} & \texttt{H?LsuNn} & \texttt{H?Lur\textasciicircum{}\{} & \texttt{H?Lutz\}} & \texttt{H?LvCv\textbackslash{}} & \texttt{H?MQ\textasciitilde{}Hn} & \texttt{H?MRJv]} & \texttt{H?MRZqt}\\
\texttt{H?MVBNX} & \texttt{H?MZMfj} & \texttt{H?MZMfz} & \texttt{H?M]blm} & \texttt{H?M]fLz} & \texttt{H?M]rLt} & \texttt{H?MamRp} & \texttt{H?Mamt\}} & \texttt{H?Mamzi} & \texttt{H?ManrM}\\
\texttt{H?Ma\textasciitilde{}b\{} & \texttt{H?Meax\}} & \texttt{H?Mfaz\textbackslash{}} & \texttt{H?Mishb} & \texttt{H?MqQeb} & \texttt{H?MurzN} & \texttt{H?M\textasciitilde{}Avp} & \texttt{H?NVI\}z} & \texttt{H?Na\textbar{}rr} & \texttt{H?Nebv\}}\\
\texttt{H?OXk\textasciicircum{}o} & \texttt{H?O\}@eX} & \texttt{H?SvdX\textbar{}} & \texttt{H?S\textbar{}NfY} & \texttt{H?S\textbar{}nr]} & \texttt{H?S\textasciitilde{}by\}} & \texttt{H?TeP\}v} & \texttt{H?Tt\textasciigrave{}ul} & \texttt{H?TttZr} & \texttt{H?UVhxl}\\
\texttt{H?UXnDz} & \texttt{H?U\textbackslash{}Jfj} & \texttt{H?WPmM\textbar{}} & \texttt{H?Wsqnf} & \texttt{H?W\}dfN} & \texttt{H?Z@ktt} & \texttt{H?ZP\textbar{}rr} & \texttt{H?[u]m\textbar{}} & \texttt{H?[\textbar{}I\textasciitilde{}r} & \texttt{H?\textbackslash{}TLn]}\\
\texttt{H?]RLfL} & \texttt{H?]dan\textbackslash{}} & \texttt{H?]dan\textbar{}} & \texttt{H?\textasciicircum{}C\textasciigrave{}nf} & \texttt{H?cvAxm} & \texttt{H?dRTL\textbar{}} & \texttt{H?dmPlr} & \texttt{H?hUPnt} & \texttt{H?hXcdb} & \texttt{H?hqtny}\\
\texttt{H?iRY\textbar{}\{} & \texttt{H?lAlGv} & \texttt{H?lSbK\textasciitilde{}} & \texttt{H?nretn} & \texttt{H?qaxzo} & \texttt{H?szdne} & \texttt{H?szfNu} & \texttt{H?tdPlV} & \texttt{H?wylfv} & \texttt{H@DMTL\textbar{}}\\
\texttt{H@EAW\textasciitilde{}f} & \texttt{H@EeQxm} & \texttt{H@EfA\textasciicircum{}X} & \texttt{H@HKvj\}} & \texttt{H@HSZv]} & \texttt{H@IQU?\textasciitilde{}} & \texttt{H@LKMc\textasciitilde{}} & \texttt{H@LK\textasciitilde{}Hv} & \texttt{H@LK\textasciitilde{}Jv} & \texttt{H@LK\textasciitilde{}J\textasciitilde{}}\\
\texttt{H@LK\textasciitilde{}bc} & \texttt{H@LK\textasciitilde{}g\textasciitilde{}} & \texttt{H@LLMd\textasciicircum{}} & \texttt{H@LMfM\textasciicircum{}} & \texttt{H@LNCn[} & \texttt{H@LNCn\textasciicircum{}} & \texttt{H@LVB]\textasciicircum{}} & \texttt{H@LVC\textbackslash{}\textbar{}} & \texttt{H@L[Vff} & \texttt{H@L\textbackslash{}UNR}\\
\texttt{H@L\textbackslash{}nRN} & \texttt{H@L\textbackslash{}vJN} & \texttt{H@Lu]t\textasciitilde{}} & \texttt{H@Lu]v\textasciitilde{}} & \texttt{H@MIIff} & \texttt{H@MI\textasciitilde{}Jr} & \texttt{H@MJMd\textasciicircum{}} & \texttt{H@MJMf\textasciicircum{}} & \texttt{H@MJ]\textasciitilde{}u} & \texttt{H@MNIzp}\\
\texttt{H@MQ\textbackslash{}hj} & \texttt{H@MQ\textasciicircum{}ji} & \texttt{H@MRI\textasciicircum{}x} & \texttt{H@MazZR} & \texttt{H@MiuK\textasciitilde{}} & \texttt{H@MiuNv} & \texttt{H@MiuN\textasciitilde{}} & \texttt{H@Mj]nZ} & \texttt{H@Mma\textasciicircum{}p} & \texttt{H@MuZzZ}\\
\texttt{H@OP[\textasciicircum{}[} & \texttt{H@Op\{\textasciitilde{}k} & \texttt{H@Or[zX} & \texttt{H@OsZv]} & \texttt{H@Os]p\}} & \texttt{H@Os]re} & \texttt{H@OsuXm} & \texttt{H@OszZZ} & \texttt{H@Oy\textasciitilde{}rm} & \texttt{H@O\{ZfZ}\\
\texttt{H@O\{uNg} & \texttt{H@O\{uNw} & \texttt{H@O\{vN]} & \texttt{H@PLSl\{} & \texttt{H@P\textbackslash{}TbN} & \texttt{H@P\textbackslash{}dT\textbar{}} & \texttt{H@P\textbackslash{}dV\{} & \texttt{H@P\textbackslash{}\textasciitilde{}p\textasciitilde{}} & \texttt{H@P]TMw} & \texttt{H@PstV\textbackslash{}}\\
\texttt{H@Pstvk} & \texttt{H@Ps\textasciitilde{}Qz} & \texttt{H@Pts\textasciitilde{}l} & \texttt{H@Q@\textbar{}PT} & \texttt{H@QH\textasciitilde{}h\textasciicircum{}} & \texttt{H@QJcze} & \texttt{H@QJkzf} & \texttt{H@QRtZ\{} & \texttt{H@QRtZ\textasciitilde{}} & \texttt{H@QRt\textasciicircum{}\{}\\
\texttt{H@QTrzk} & \texttt{H@QX\textasciitilde{}V\{} & \texttt{H@QYtVc} & \texttt{H@QZRm\}} & \texttt{H@QZbU\textbar{}} & \texttt{H@QZdT\textbar{}} & \texttt{H@QZdV\textasciitilde{}} & \texttt{H@QZd\textasciicircum{}Y} & \texttt{H@QZtze} & \texttt{H@Q\textbackslash{}Q\textbar{}u}\\
\texttt{H@Q\textbackslash{}bT\textbar{}} & \texttt{H@Q\textbackslash{}bvk} & \texttt{H@Q\textasciicircum{}Bu\{} & \texttt{H@Q\textasciicircum{}LpZ} & \texttt{H@Q\textasciicircum{}Rm\textbar{}} & \texttt{H@Q\textasciicircum{}Ry\}} & \texttt{H@Q\textasciicircum{}ThZ} & \texttt{H@Qjeu]} & \texttt{H@Qje\textasciitilde{}\}} & \texttt{H@Qjmv\textbar{}}\\
\texttt{H@Qm\textasciigrave{}v\textbackslash{}} & \texttt{H@Qmbu]} & \texttt{H@Qmrl\textasciitilde{}} & \texttt{H@Qnay\textasciicircum{}} & \texttt{H@QqtVN} & \texttt{H@QsQvf} & \texttt{H@Qsrvk} & \texttt{H@RHsvc} & \texttt{H@RP\{\textasciitilde{}j} & \texttt{H@RSp\textasciitilde{}j}\\
\texttt{H@R\textbackslash{}Ru\textasciitilde{}} & \texttt{H@R\textasciigrave{}uu]} & \texttt{H@Rcru\}} & \texttt{H@Sp\{\textasciicircum{}d} & \texttt{H@Sr[zf} & \texttt{H@SsZNZ} & \texttt{H@Ss\textasciicircum{}L\}} & \texttt{H@Tctnk} & \texttt{H@Tctnl} & \texttt{H@Tc\textbar{}Zp}\\
\texttt{H@Tc\textbar{}ze} & \texttt{H@Tdk\textasciitilde{}l} & \texttt{H@Thmun} & \texttt{H@Tktne} & \texttt{H@TnCmx} & \texttt{H@Ts\textbar{}\textasciicircum{}f} & \texttt{H@TtS\textasciitilde{}e} & \texttt{H@TtTfL} & \texttt{H@TtU\textasciicircum{}u} & \texttt{H@TuTel}\\
\texttt{H@TvCul} & \texttt{H@UBlZ\textasciitilde{}} & \texttt{H@UCjZf} & \texttt{H@UHnD\textasciicircum{}} & \texttt{H@UJH\textasciitilde{}u} & \texttt{H@UJdN\textasciicircum{}} & \texttt{H@UJkzf} & \texttt{H@UJlZr} & \texttt{H@ULbN\textbar{}} & \texttt{H@UTB\textasciitilde{}m}\\
\texttt{H@U\textbackslash{}Rnf} & \texttt{H@Ucrnk} & \texttt{H@UmbK\textasciitilde{}} & \texttt{H@Up\textasciitilde{}vm} & \texttt{H@Uu\textasciicircum{}V\textbar{}} & \texttt{H@UurYn} & \texttt{H@Ux\textasciitilde{}fj} & \texttt{H@VCp\textasciitilde{}f} & \texttt{H@Vb[\}z} & \texttt{H@XSt\textasciicircum{}V}\\
\texttt{H@XStnk} & \texttt{H@XTSl\textasciicircum{}} & \texttt{H@XTSn[} & \texttt{H@XTSn\textasciicircum{}} & \texttt{H@XTSn\textasciitilde{}} & \texttt{H@XTc\textasciitilde{}k} & \texttt{H@X\textbackslash{}e\textasciicircum{}u} & \texttt{H@X\textbackslash{}l\textasciitilde{}y} & \texttt{H@YLinX} & \texttt{H@YVA\}\{}\\
\texttt{H@YYnVv} & \texttt{H@YY\textasciitilde{}qv} & \texttt{H@Y\textasciicircum{}A\textasciitilde{}v} & \texttt{H@Yamq]} & \texttt{H@Yamu]} & \texttt{H@YszzZ} & \texttt{H@ZP\textbar{}v\textasciicircum{}} & \texttt{H@Zakvz} & \texttt{H@\textbackslash{}Sl\textasciicircum{}e} & \texttt{H@]UNNy}\\
\texttt{H@]u]n\textasciitilde{}} & \texttt{H@\textasciicircum{}CjNw} & \texttt{H@\textasciigrave{}I\textbar{}g\textasciitilde{}} & \texttt{H@\textasciigrave{}Jlz]} & \texttt{H@\textasciigrave{}Kb?\textasciitilde{}} & \texttt{H@\textasciigrave{}ZLvw} & \texttt{H@\textasciigrave{}ZLvx} & \texttt{H@\textasciigrave{}\textasciicircum{}B\}\}} & \texttt{H@\textasciigrave{}\textasciicircum{}Ry\}} & \texttt{H@\textasciigrave{}a\textbar{}rN}\\
\texttt{H@biYsz} & \texttt{H@biptv} & \texttt{H@caMLy} & \texttt{H@dTJV[} & \texttt{H@dTRnk} & \texttt{H@dZd\textasciicircum{}f} & \texttt{H@dtIvh} & \texttt{H@dvA]x} & \texttt{H@dvA\}n} & \texttt{H@f\textasciicircum{}Rzr}\\
\texttt{H@fbq\textasciitilde{}\textbar{}} & \texttt{H@hKaLr} & \texttt{H@hUb]\textasciicircum{}} & \texttt{H@h]Dtu} & \texttt{H@huI\}z} & \texttt{H@luvJ\textasciitilde{}} & \texttt{H@meI\textbar{}v} & \texttt{H@pi\textasciigrave{}eN} & \texttt{H@plclZ} & \texttt{H@qaju\}}\\
\texttt{H@qbay]} & \texttt{H@silne} & \texttt{H@uJHl\textasciitilde{}} & \texttt{HAEf@\textasciicircum{}X} & \texttt{HAIZvZu} & \texttt{HAKky\textasciitilde{}f} & \texttt{HAMSRNf} & \texttt{HAMruYn} & \texttt{HAW\textbar{}K\textasciitilde{}q} & \texttt{HAYPmYj}\\
\texttt{HAYsp\textasciitilde{}m} & \texttt{HA\textasciigrave{}hrny} & \texttt{HActJVN} & \texttt{HAgZK\textasciitilde{}u} & \texttt{HBF@[\textasciicircum{}b} & \texttt{HBOHl]]} & \texttt{HBOklVN} & \texttt{HBOk\}Qd} & \texttt{HBQ\textasciigrave{}[vN} & \texttt{HBaX]Tf}\\
\texttt{HBdLJ]\}} & \texttt{HBgQK\textasciicircum{}m} & \texttt{HCHasx\}} & \texttt{HCQj\textasciigrave{}\textasciitilde{}x} & \texttt{HCSp\textasciicircum{}N]} & \texttt{HChhqnV} & \texttt{HKoq\textasciigrave{}]\textasciicircum{}} &  &  & \\
\addlinespace[2pt]
\multicolumn{10}{@{}l}{\bfseries 10 vertices (146 records)}\\[-1pt]
\texttt{I??@\textbar{}X[\textbar{}?} & \texttt{I??EdPKL?} & \texttt{I??IhahrO} & \texttt{I??MTHSM?} & \texttt{I??V@plfO} & \texttt{I??XCdhrg} & \texttt{I??XQel\{O} & \texttt{I??XStevG} & \texttt{I??XYmj\textbar{}?} & \texttt{I??XaUlro}\\
\texttt{I??XfPMlg} & \texttt{I??Y\textbackslash{}xyng} & \texttt{I??Y\textasciigrave{}Ukpo} & \texttt{I??[i]qYW} & \texttt{I??gr]\textasciicircum{}\textbar{}\_} & \texttt{I??kRpUfg} & \texttt{I??moxdmW} & \texttt{I??orRBbg} & \texttt{I??pQrEeo} & \texttt{I??pSXRqg}\\
\texttt{I??u@TWfW} & \texttt{I?AJ@rOF\_} & \texttt{I?APYlheg} & \texttt{I?AQV?wNg} & \texttt{I?AnMhXMg} & \texttt{I?CYPNUmO} & \texttt{I?CYhZIkO} & \texttt{I?CYhZIkW} & \texttt{I?C\textasciigrave{}jZLug} & \texttt{I?C\textasciigrave{}mZYZo}\\
\texttt{I?C\textasciigrave{}mrKZg} & \texttt{I?CaZI\textbackslash{}m\_} & \texttt{I?CaZI\textasciicircum{}mo} & \texttt{I?Ca[xfm\_} & \texttt{I?CeaylNw} & \texttt{I?ChmFXZw} & \texttt{I?ChmVsVo} & \texttt{I?CiClyfw} & \texttt{I?CidrE\textasciitilde{}\_} & \texttt{I?CieKyjW}\\
\texttt{I?Cil?\textbar{}qo} & \texttt{I?ClQg\{ow} & \texttt{I?Clzz[pw} & \texttt{I?Clzzsvw} & \texttt{I?Cmmu\textbar{}\textasciicircum{}W} & \texttt{I?CqkXfiw} & \texttt{I?Cup\textbackslash{}l\textasciitilde{}G} & \texttt{I?CuqW\textasciitilde{}zW} & \texttt{I?CyvC\}jW} & \texttt{I?D?x]rz?}\\
\texttt{I?D@HjImW} & \texttt{I?D@h]Zz?} & \texttt{I?DHeA\_Fw} & \texttt{I?DHmQp\textasciitilde{}w} & \texttt{I?DL@hIcw} & \texttt{I?DPXvlvo} & \texttt{I?DPx]\textbar{}yo} & \texttt{I?DXM\_ziw} & \texttt{I?DXXrRww} & \texttt{I?D[LdiYw}\\
\texttt{I?D[tdn\textasciitilde{}G} & \texttt{I?DuXwzuw} & \texttt{I?E@mpmZo} & \texttt{I?EIhMh[W} & \texttt{I?E\textasciigrave{}Yhz\}o} & \texttt{I?EaGdh\}W} & \texttt{I?Eq@SZwW} & \texttt{I?EzvbNzo} & \texttt{I?FHQcvmo} & \texttt{I?GIKpskg}\\
\texttt{I?GRKXVeo} & \texttt{I?GSHl[q\_} & \texttt{I?GSZHRfw} & \texttt{I?GUQGtmO} & \texttt{I?GXa\textasciicircum{}Uyo} & \texttt{I?GXmV[\textbackslash{}o} & \texttt{I?GqO\textasciitilde{}Uyo} & \texttt{I?HGiqbdg} & \texttt{I?HPco\textasciicircum{}r\_} & \texttt{I?HSzW\textasciitilde{}\}W}\\
\texttt{I?HYTeu\}G} & \texttt{I?H[rqfv\_} & \texttt{I?Hi\textasciigrave{}uZtg} & \texttt{I?IYu\_\textasciitilde{}\textasciitilde{}w} & \texttt{I?IYv?\}\{O} & \texttt{I?I\_\_lJug} & \texttt{I?KYkLhjw} & \texttt{I?KlIzQsw} & \texttt{I?LKHcVqW} & \texttt{I?LT\textasciicircum{}\textasciigrave{}\textbar{}ng}\\
\texttt{I?Lce?nFo} & \texttt{I?LeC\_NDw} & \texttt{I?MPRl]bw} & \texttt{I?MQLT]Zw} & \texttt{I?MZVm\}\textasciicircum{}g} & \texttt{I?OLPhVlO} & \texttt{I?OUH[\{cw} & \texttt{I?Pl\textasciigrave{}eXVW} & \texttt{I?QJC\{v\textbackslash{}g} & \texttt{I?Q\textasciigrave{}dt]\textasciicircum{}g}\\
\texttt{I?RHdc]]W} & \texttt{I?RgpUrTo} & \texttt{I?SU@YeLo} & \texttt{I?W\textbackslash{}LD\textbackslash{}Tw} & \texttt{I?XPdeN\textasciitilde{}w} & \texttt{I?]CGcfTW} & \texttt{I?\textasciigrave{}E@osF\_} & \texttt{I?aQxxjz?} & \texttt{I?cHIdsbw} & \texttt{I?cb?xUhw}\\
\texttt{I?ccIhi[o} & \texttt{I?hRtg\textasciitilde{}\textasciitilde{}G} & \texttt{I?kRAKfsW} & \texttt{I?l?\textasciigrave{}KVxg} & \texttt{I?s\_hdFxW} & \texttt{I@?IShff\_} & \texttt{I@C[ZNJz?} & \texttt{I@H\textbackslash{}IVPew} & \texttt{I@IYrUvvO} & \texttt{I@MbQqVpw}\\
\texttt{I@OSGljjO} & \texttt{I@Pu[wzmw} & \texttt{I@QPKTZZo} & \texttt{I@\_qi\}myW} & \texttt{I@\textasciigrave{}LAovNw} & \texttt{I@d@CLN\textbackslash{}o} & \texttt{I@dYSdnlw} & \texttt{I@iQa[\}Zo} & \texttt{IAI\_zq]vo} & \texttt{IB?[U]mZ\_}\\
\texttt{IB\_OXLJjW} & \texttt{ICHSASnNw} & \texttt{ICH\textasciigrave{}qy]vo} & \texttt{ICKaC\textbackslash{}UXw} & \texttt{ICOo\textbar{}Dluw} & \texttt{ICTKHKzZw} & \texttt{ICTSPkn\textasciitilde{}\_} & \texttt{IDCjKXJrw} & \texttt{IDQ@a[mvW} & \texttt{IGQ]@kyeW}\\
\texttt{IHqHGkZRw} & \texttt{IQQ@\textasciigrave{}OVBo} & \texttt{I\textasciigrave{}?LSleUW} & \texttt{I\textasciigrave{}hQC?ZHw} & \texttt{Io?Y\textasciigrave{}QF]G} & \texttt{Io@y\textasciigrave{}SZ\textasciigrave{}w} &  &  &  & \\
\addlinespace[2pt]
\multicolumn{10}{@{}l}{\bfseries 11 vertices (10 records)}\\[-1pt]
\texttt{J???[pQjJh?} & \texttt{J??BaqDJKz\_} & \texttt{J??OXaaRUF\_} & \texttt{J??\_\textbar{}\textbar{}\}veV\_} & \texttt{J??czQsV]\}\_} & \texttt{J?@ipqo\textasciigrave{}\}\textasciicircum{}?} & \texttt{J?CIhTDkNJ\_} & \texttt{J?CP\textasciigrave{}NHrb\textasciicircum{}?} & \texttt{J?CqcTDijJ\_} & \texttt{J?M@IHBF\{n\_}\\
\addlinespace[2pt]
\multicolumn{10}{@{}l}{\bfseries 12 vertices (4 records)}\\[-1pt]
\texttt{K???@LQi?\{X\_} & \texttt{K???@S[hAcXo} & \texttt{K?@W\textasciigrave{}?b\_oVx]} & \texttt{K?CcIaLKi]ly} &  &  &  &  &  & \\
\addlinespace[2pt]
}
 
\newcommand{\obstructionpanel}[2]{\begin{subfigure}[t]{0.32\textwidth}
    \centering
    \scalebox{1.08}{\obstructiongraph{#1}}
    \caption{$#2$}
  \end{subfigure}}

\title{The Excluded Vertex-Minors and Pivot-Minors\\for Rank-Width at Most Two}
\author{Sang-il Oum\thanks{Supported by the Institute for Basic Science (IBS-R029-C1).}\\
\small Discrete Mathematics Group, Institute for Basic Science (IBS),
Daejeon, South Korea\\
\small Department of Mathematical Sciences, KAIST, Daejeon, South Korea\\ 
\small \texttt{sangil@ibs.re.kr}}
\date{\today}

\begin{document}
\maketitle

\begin{abstract}
We determine both the excluded vertex-minors and the excluded pivot-minors for the class
of graphs of rank-width at most two.  Up to local equivalence and graph isomorphism, there
are exactly 25 excluded vertex-minors: 1 graph on 8 vertices, 18 on 9 vertices,
and 6 on 10 vertices.  Up to pivot equivalence and graph isomorphism, there are exactly
609 excluded pivot-minors: 2 on 8 vertices, 447 on 9 vertices, 146 on 10
vertices, 10 on 11 vertices, and 4 on 12 vertices.  No excluded vertex-minor
occurs on 11--16 vertices, and no excluded pivot-minor occurs on 13--16 vertices;
the author's 16-vertex bound makes both lists complete.

The proof is computer-assisted.  Instead of enumerating all graphs, we reverse the
one-vertex reduction theorem for prime graphs.  For each $n$, we retain exactly the prime $n$-vertex
graphs of rank-width at most two, modulo local equivalence and isomorphism, and extend
those graphs by one vertex.  Local-equivalence classes are identified by an exact canonical
key obtained from the associated isotropic system, the binary row space of $[I\mid A(G)]$.  
A restricted version of the same key classifies
pivot equivalence exactly.  The vertex-minor and pivot-minor computations examine,
respectively, more than $9.0\times 10^{10}$ and $4.9\times 10^{11}$ prime extensions in
their 16-vertex final layers.  
\end{abstract}



\section{Introduction}

Rank-width is a width parameter of graphs introduced by Oum and Seymour
\cite{OS2004}.  It measures the complexity of edge cuts by the binary rank of
their bipartite adjacency matrices.  Rank-width is invariant under local complementation
and does not increase under vertex deletion; it is therefore monotone under taking
vertex-minors and pivot-minors \cite{Oum2004}.  For each fixed $k$, the graphs of rank-width at most $k$
are characterized by finitely many excluded vertex-minors and finitely many excluded
pivot-minors \cite{Oum2004a,Oum2004}.  Explicit
lists, however, are difficult to obtain even for small $k$.

Oum \cite{Oum2004} proved that, for every integer $k$, every excluded
vertex-minor or pivot-minor for rank-width at most $k$ has at most $(6^{k+1}-1)/5$ vertices.  
At the
first nontrivial level, 
by Bouchet's theorem (\zcref{thm:bouchet-chain}), $C_5$ is the unique excluded vertex-minor for rank-width at most one, 
and by Allys's theorem (\zcref{thm:allys-chain}), one can deduce easily that $C_5$
and~$C_6$ are the two excluded pivot-minors for rank-width at most one; see also
\cite[Section~4.7]{Oum2016}.  Indeed, graphs of rank-width at most one are precisely
the distance-hereditary graphs \cite[Proposition~7.3]{Oum2004}, and, by
\zcref{prop:primesubgraph}, these are exactly the graphs with no prime induced subgraph
on at least five vertices.

The case $k=2$ became finite enough for a complete search when the author  
improved the bound of $(6^{3}-1)/5=43$ to $16$ in \cite[Theorem~7.6]{Oum2020}.
A direct isomorph-free enumeration of all graphs with at most 16 vertices remains far too large.
The decisive change is to enumerate only a
\emph{prime frontier}: prime graphs of rank-width at most two, one representative per
local-equivalence and isomorphism class.  The reverse step is justified by Bouchet's
chain theorem, stated as \zcref{thm:bouchet-chain}.  We reverse this reduction by adding one vertex in
all possible ways, retain the prime extensions of rank-width at most two as the next
frontier, and test the remaining extensions for vertex-minor minimality.

To quotient these frontiers exactly, we encode the associated isotropic system 
of Bouchet \cite{Bouchet1987a}
by a
colored incidence graph formed from low-weight words
so that if a graph $G$ is locally equivalent to a graph isomorphic to $H$,
then both $G$ and $H$ are encoded into the same key.
This code-graph construction is
adapted from Danielsen and Parker~\cite{DP2009}, who extend Östergård's method for
classifying linear codes~\cite{Ostergard2002}.  Our adaptive variant tests the
automorphisms of the truncated incidence graph on the full isotropic space before
accepting a weight cutoff; \zcref{sec:lc-key} gives the construction and its correctness
proof.

The pivot-minor computation requires a finer frontier because a local-equivalence class
may contain many pivot-equivalence classes.  Beginning with graphs on five vertices, we use
Allys's chain theorem (\zcref{thm:allys-chain}) to extend complete pivot frontiers
successively to graphs on at most~12 vertices, inserting the exceptional cycle classes
explicitly.  As a cross-check for graphs on 12 vertices, we directly split every complete
local-equivalence class: we enumerate its
labeled local-complementation orbit, quotient first by graph isomorphism, and deduplicate
by an exact pivot-equivalence key.

The later layers search only for excluded pivot-minors rather than complete pivot frontiers.
For candidate graphs on 13, 14, 15, or 16 vertices, a stronger chain theorem of the author
(\zcref{thm:oum-chain}) allows us to restrict to prime
$3^{+3}$-rank-connected parents.  For the 13-vertex search, we filter the complete
pivot frontier $\cPiv_{12}$ directly.  For the 14--16-vertex searches, we select the
relevant local-equivalence classes from the complete local-equivalence frontiers and
split exactly those classes into pivot-equivalence classes.
For candidate graphs on 15 or 16 vertices, we additionally require a four-vertex set $X$
of cut-rank $\rho(X)\leq2$; \zcref{prop:excluded-not-three-plus-one} justifies this filter.
Each candidate is tested by deletion and by one pivot-deletion at every vertex;
one incident edge suffices because all choices yield pivot-equivalent graphs.  A
restriction of the same classifier gives an exact pivot-equivalence
key, so neither construction relies on probabilistic hashing.  Finally, an independent
subset dynamic program verifies the rank-width and pivot-minimality of all 609 reported
graphs without using the search implementation.

Our main result is the following.  In $O_{n,i}$, the first subscript is the number of
vertices and the second is the position in the graph6 list of \zcref{app:graph6}.

\begin{theorem}\label{thm:main}
Let $G$ be a graph.  Then $\rw(G)\leq 2$ if and only if none of the 25 graphs
\[
  O_{8,1},\quad O_{9,1},\ldots,O_{9,18},\quad
  O_{10,1},\ldots,O_{10,6}
\]
shown in \zcref{fig:o8,fig:o9,fig:o10} is a vertex-minor of $G$.
The 25 graphs are pairwise inequivalent under local complementation and graph
isomorphism.  Each has rank-width three and every proper vertex-minor has rank-width at
most two.
\end{theorem}

The corresponding pivot-minor list is much larger; its 609 graph6 records are given in
\zcref{tab:pivot-graph6}.

\begin{theorem}\label{thm:pivot-main}
Up to pivot equivalence and graph isomorphism, there are exactly 609 excluded
pivot-minors for the class of graphs of rank-width at most two.  Their distribution by the number of vertices
is
\[
\begin{array}{@{}crrrrrrrrrrr@{}}
\toprule
|V|&6&7&8&9&10&11&12&13&14&15&16\\
\midrule
\text{classes}&0&0&2&447&146&10&4&0&0&0&0\\
\bottomrule
\end{array}.
\]
Consequently, a graph has rank-width at most two if and only if it has no pivot-minor
isomorphic to one of the 609 cataloged graphs.
\end{theorem}

\section{Preliminaries}\label{sec:prelim}

\paragraph{Vertex-minors and pivot-minors.}
All graphs are finite and simple.  
For $v\in V(G)$, the \emph{local complementation} $G*v$ toggles adjacency between every
pair of distinct neighbors of $v$.  Graphs related by a sequence of local
complementations are \emph{locally equivalent}.  If $uv\in E(G)$, then
\[
  G\wedge uv=G*u*v*u=G*v*u*v
\]
is the graph obtained by \emph{pivoting} the edge $uv$.  A graph $H$ is a
\emph{vertex-minor} of $G$ if $H$ is an induced subgraph of a graph locally equivalent
to $G$.  Two graphs on the same vertex set are \emph{pivot-equivalent} if one is obtained
from the other by a sequence of edge pivots.  A graph $H$ is a \emph{pivot-minor} of $G$
if it is an induced subgraph of a graph pivot-equivalent to $G$.
We say that $G$ and $H$ are \emph{locally equivalent up to isomorphism}, respectively
\emph{pivot-equivalent up to isomorphism}, if $G$ is locally equivalent, respectively
pivot-equivalent, to a graph isomorphic to $H$.

\paragraph{Cut-rank functions and rank-decompositions of graphs.}
For an $X\times Y$ matrix $M$ and subsets $X'\subseteq X$ and $Y'\subseteq Y$, 
we write $M[X',Y']$ to denote its $X'\times Y'$ submatrix. 
Let $A(G)$ be the adjacency matrix of a graph $G$ over
$\Ftwo$.  For $X\subseteq V(G)$, the \emph{cut-rank} is
\[
  \rho_G(X)=\rank_{\Ftwo} A(G)[X,V(G)\setminus X].
\]
It is symmetric: $\rho_G(X)=\rho_G(V(G)\setminus X)$. More importantly, it is \emph{submodular}: 
\[ 
\rho_G(X)+\rho_G(Y)\ge \rho_G(X\cup Y)+\rho_G(X\cap Y)\quad \text{ for all }X,Y\subseteq V(G).
\] 

A \emph{rank-decomposition} of a graph $G$ with at least two vertices is a pair $(T,\mu)$,
where $T$ is a subcubic tree and $\mu$ is a bijection from $V(G)$ to the leaves of $T$.
Every edge of $T$ displays a bipartition $(X,V(G)\setminus X)$; its width is
$\rho_G(X)$.  The width of $(T,\mu)$ is the maximum width of its edges, and the
\emph{rank-width} $\rw(G)$ is the minimum width of a rank-decomposition.  The standard
convention is that every graph with at most one vertex has rank-width zero.

A \emph{split} of $G$ is a partition $(X,V(G)\setminus X)$ with both parts of size at
least two and $\rho_G(X)\leq 1$.  A connected graph with no split is \emph{prime}.

\begin{proposition}[Hlin{\v e}n\'y, Oum, Seese, and Gottlob~{\cite[Theorem 4.3]{HOSG2006}}]\label{prop:primesubgraph}
  The rank-width of a graph is equal to the maximum rank-width of all of its prime induced subgraphs.
\end{proposition}

An \emph{excluded vertex-minor} for rank-width at most two is a graph $G$ with
$\rw(G)>2$ such that every proper vertex-minor of $G$ has rank-width at most two.
An \emph{excluded pivot-minor} is defined in the same way with ``vertex-minor'' replaced
by ``pivot-minor.''  Every excluded vertex-minor and every excluded pivot-minor has
rank-width exactly three: deleting any vertex gives rank-width at most two, and adding
one vertex increases rank-width by at most one.
Moreover, all of them are prime.  Indeed, if such a graph were not prime, then
\zcref{prop:primesubgraph} would give a proper prime induced subgraph of rank-width
three.  That induced subgraph would be a proper vertex-minor, and also a proper
pivot-minor, contradicting the corresponding minimality assumption.

Local complementation preserves the cut-rank function~\cite{Oum2004}.
Therefore, local complementation preserves connectivity and splits, 
so primeness is a
property of a local-equivalence class, see Bouchet~\cite[Lemma 2.2]{Bouchet1987b}.
Moreover, we deduce the following fact, which will be used repeatedly.

\begin{proposition}[Oum \cite{Oum2004}]\label{prop:monotone}
Rank-width is invariant under local complementation and is nonincreasing under vertex
deletion.  Consequently it is nonincreasing under taking vertex-minors.
\end{proposition}

We recall a tool for prescribing cuts in an optimal decomposition.  Titanic partitions
originate in the branch-width theory of Robertson and Seymour; the form needed here is
the following lemma of Hlin{\v e}n{\'y} and Oum.  If $f$ is a
symmetric submodular function on $2^V$, its branch-width $\bw(f)$ is defined in the
same way as rank-width, with $f$ in place of $\rho_G$; thus
$\rw(G)=\bw(\rho_G)$.  A set $Q\subseteq V$ is \emph{titanic with respect to $f$} if,
for every partition $(Q_1,Q_2,Q_3)$ of $Q$, some $i\in\{1,2,3\}$ satisfies
$f(Q_i)\geq f(Q)$.  A partition $\mathcal P$ of $V$ is \emph{titanic} if each of its
parts is titanic.  Its width is $\max_{Q\in\mathcal P}f(Q)$.  Define the symmetric
submodular function $f^{\mathcal P}$ on $2^{\mathcal P}$ by
\[
 f^{\mathcal P}(\mathcal A)=f\left(\bigcup_{Q\in\mathcal A}Q\right).
\]

\begin{lemma}[Hlin{\v e}n{\'y} and Oum {\cite[Lemma~3.4]{HO2006}}]
\label{lem:titanic-partition}
Let $f$ be a symmetric submodular function on $2^V$ with $\bw(f)\leq k$.  If
$\mathcal P$ is a titanic partition of width at most $k$, then
$\bw(f^{\mathcal P})\leq k$.
\end{lemma}

\paragraph{Chain theorems on vertex-minors and pivot-minors.}
For a vertex $v$, choose any neighbor $u$ when one exists and write
$G/v=G\wedge uv\setminus v$; different choices give pivot-equivalent graphs.  If $v$
is isolated, put $G/v=G\setminus v$.  In fact, if $u$ and $w$ are neighbors of $v$, then
\[
G\wedge uv=(G\wedge wv)\wedge uw;
\]
hence the two resulting graphs after deleting $v$
are pivot-equivalent \cite[Proposition~2.5]{Oum2004}.  Thus one incident edge per
nonisolated vertex represents all one-vertex pivot-minors of this form.

\begin{proposition}[{Bouchet \cite[(9.2)]{Bouchet1988} and Fon-Der-Flaass~\cite[Corollary 4.3]{FonDerFlaass1988}}]\label{prop:three-reductions}
  Let $v$ be a vertex of a graph~$G$.
Every vertex-minor of $G$ with vertex set $V(G)\setminus\{v\}$ is locally equivalent to
one of
\[
  G\setminus v,\qquad G*v\setminus v,\qquad G/v.
\]
\end{proposition}

We call these three graphs the \emph{elementary reductions} of $G$ at $v$.

The following proposition follows easily from the connection between the pivot operation of graphs
and the principal pivot operation of matrices introduced by Tucker~\cite{Tucker1960}.
A direct proof appeared in Dabrowski, Dross, Jeong, Kant\'e, Kwon, Oum, and Paulusma~\cite{DDJKKOP2023}.

\begin{proposition}[See~{\cite[Lemma 2.3]{DDJKKOP2023}}]\label{prop:two-pivot-reductions}
    Let $v$ be a vertex of a graph~$G$.
Every pivot-minor of~$G$ with vertex set $V(G)\setminus\{v\}$ is pivot-equivalent to one of 
  $G\setminus v$ or $G/v$.
\end{proposition}

The two chain theorems used to reverse the generation are as follows.  We state their
hypotheses on the number of vertices and their exceptions explicitly.

\begin{theorem}[Bouchet~{\cite[Theorem, p.~244]{Bouchet1987b}}]
\label{thm:bouchet-chain}
Every prime graph~$G$ with at least six vertices has a prime vertex-minor on $\abs{V(G)}-1$ vertices.
\end{theorem}

\begin{theorem}[Allys~{\cite[Theorem~4.3]{Allys1994}}]
\label{thm:allys-chain}
Every prime graph $G$ with at least six vertices has a prime pivot-minor on
$\abs{V(G)}-1$ vertices, unless $G$ is pivot-equivalent to a cycle.
\end{theorem}

\paragraph{Higher rank connectivity.}

For integers $k$ and $\ell$, a graph $G$ is \emph{$k^{+\ell}$-rank-connected} if
\[
  \rho_G(X)<k\quad\Longrightarrow\quad
  \min\{|X|,|V(G)\setminus X|\}<k+\ell
\]
for every $X\subseteq V(G)$ \cite{Oum2020}.  A graph is \emph{$k$-rank-connected} if it
is $m^{+0}$-rank-connected for every integer $1\leq m\leq k$.  In particular, the
$2$-rank-connected graphs are precisely the connected graphs with no split.

\begin{proposition}[Oum {\cite[Proposition~7.3]{Oum2020}}]
\label{prop:excluded-connectivity}
Let $G$ be a graph with $\rw(G)>2$.  If
$\rw(G\setminus v)<\rw(G)$ and $\rw(G/v)<\rw(G)$ for every vertex $v$ of $G$, then
$G$ is prime and $3^{+2}$-rank-connected.
\end{proposition}

\begin{theorem}[Oum {\cite[Theorem~7.6]{Oum2020}}]
\label{thm:excluded-pivot-bound}
If $G$ has rank-width three and every proper pivot-minor of~$G$ has rank-width at most
two, then $|V(G)|\leq16$.
\end{theorem}

\begin{theorem}[Oum {\cite[Theorem~4.1]{Oum2020}}]
\label{thm:oum-chain}
Every prime $3^{+2}$-rank-connected graph~$G$ with at least ten vertices has a prime
$3^{+3}$-rank-connected pivot-minor on $\abs{V(G)}-1$ vertices.
\end{theorem}

\begin{proposition}[Oum {\cite[Proposition~7.7]{Oum2020}}]
\label{prop:excluded-not-three-plus-one}
If $G$ has rank-width three, every proper pivot-minor of~$G$ has rank-width at most two,
and $G$ is $3^{+1}$-rank-connected, then $|V(G)|\leq14$.
\end{proposition}

\section{The prime-frontier search}\label{sec:search}

For $n\geq 5$, let $\cF_n$ contain one representative of each local-equivalence and
isomorphism class of prime $n$-vertex graphs of rank-width at most two.  Let $\cO_n$
contain one representative of each excluded-vertex-minor class on $n$ vertices.  The initial frontier
$\cF_5$ consists of the unique class represented by the 5-cycle.  Indeed, Bouchet
\cite[Lemma~3.1]{Bouchet1987b} proved that every prime graph on five vertices is locally
equivalent to $C_5$, and a direct rank-decomposition gives $\rw(C_5)=2$.

\subsection{Generating one-vertex extensions}

For each $H\in\cF_n$, the program considers all neighborhoods $N\subseteq V(H)$ with
$|N|\geq 2$ and adds a new vertex $x$ with $N(x)=N$.  It discards the extension when
$x$ is a twin of an old vertex.  These elementary conditions characterize the prime
extensions of a prime parent.

\begin{lemma}[Geelen {\cite[Lemma~5.3]{Geelen1995}}]\label{lem:prime-extension}
Let $H$ be a prime graph with at least five vertices, and let $G$ be obtained from $H$
by adding a vertex $x$.  Then $G$ is prime if and only if $d_G(x)\geq 2$ and $x$ has no
twin in $G$.
\end{lemma}

The program then classifies each generated graph.  Rank-width-two children enter the
next frontier; rank-width-greater-than-two children are tested for minimality and are
never extended.  Duplicates under ordinary graph isomorphism are removed before the more expensive
local-equivalence classification of \zcref{sec:lc-key}.  We use the nauty library \cite{MP2014} to obtain the canonical labeling of a graph invariant under graph isomorphism.

\begin{algorithm}[H]
\caption{One prime-frontier layer from $n$ to $n+1$ vertices}
\label{alg:layer}
\begin{algorithmic}[1]
\Require One representative of every class in $\cF_n$
\Ensure Representatives of $\cF_{n+1}$ and $\cO_{n+1}$
\State $R\gets\varnothing$; $B\gets\varnothing$
\ForAll{$H\in\cF_n$}
  \ForAll{$N\subseteq V(H)$ with $|N|\geq 2$}
    \State form $G$ by adding $x$ with $N_G(x)=N$
    \If{$x$ has no twin in $G$}
\State replace $G$ by its ordinary isomorphism-canonical form
    \If{$\rw(G)\leq2$}
      \State insert $G$ into $R$
    \ElsIf{all three elementary reductions at every vertex have rank-width $\leq2$}
      \State insert $G$ into $B$
    \EndIf
    \EndIf
  \EndFor
\EndFor
\State $\cF_{n+1}\gets$ one representative of each local-equivalence and isomorphism class in $R$
\State $\cO_{n+1}\gets$ one representative of each local-equivalence and isomorphism class in $B$
\end{algorithmic}
\end{algorithm}

\subsection{Exact rank-width and minimality tests}\label{sec:rw-oracle}

It is known that for each fixed integer $k$, one can decide whether the rank-width of an
$n$-vertex graph is at most $k$ in time $O_k(n^2)$; see Fomin and
Korhonen~\cite[Corollary~1.3]{FK2024}.
Implementing such algorithms faithfully is complicated, and the search needs billions of
rank-width decisions.  This subsection describes the practical routine used for every
rank-width decision in this paper and proves its correctness.  Although the search
extends prime graphs, the routine is applied to the one-vertex reductions
$G\setminus v$, $G*v\setminus v$, and $G/v$, which need not be prime; the proof below
therefore assumes nothing about the input beyond connectivity, which is free because
rank-width is the maximum over connected components.  We first prove the lemmas on
which correctness rests, then present the routine as pseudocode, and finally assemble
the proof.

\paragraph{Partitioned decompositions.}
For a graph $G$ and a set $\mathcal P$ of pairwise disjoint nonempty subsets of $V(G)$,
a \emph{rank-decomposition of $(G,\mathcal P)$} is a pair $(T,\mu)$ of a subcubic tree
$T$ and a bijection $\mu$ from~$\mathcal P$ to the set of leaves of $T$.  For a subfamily
$\mathcal X\subseteq\mathcal P$ we write $\bigcup\mathcal X=\bigcup_{A\in\mathcal X}A$
and
\[
  \rho_G^{\mathcal P}(\mathcal X)
  =\rho_{G[\bigcup\mathcal P]}\Bigl(\bigcup\mathcal X\Bigr).
\]
Every edge of $T$ displays a bipartition $(\mathcal X,\mathcal P\setminus\mathcal X)$ of
the leaves, whose width is $\rho_G^{\mathcal P}(\mathcal X)$.  The width of $(T,\mu)$ is
the maximum width of its edges, and the \emph{rank-width of $(G,\mathcal P)$} is the
minimum width of a rank-decomposition of $(G,\mathcal P)$; as usual, it is $0$ when
$\abs{\mathcal P}\leq1$.  When $\mathcal P$ consists of singletons, the rank-width of
$(G,\mathcal P)$ equals the rank-width of $G[\bigcup\mathcal P]$.

For an integer $k$, a subfamily $\mathcal X\subseteq\mathcal P$ is \emph{$k$-branched}
if $\rho_G^{\mathcal P}(\mathcal X)\leq k$ and either $\abs{\mathcal X}\leq1$ or there
is a proper nonempty subfamily $\mathcal Y$ of $\mathcal X$ such that both $\mathcal Y$
and $\mathcal X\setminus\mathcal Y$ are $k$-branched.  Equivalently, $\mathcal X$ is
$k$-branched when it carries a rooted subcubic tree each of whose displayed subfamilies,
including~$\mathcal X$ itself, has $\rho_G^{\mathcal P}$ at most~$k$.  In particular the
rank-width of $(G,\mathcal P)$ is at most $k$ if and only if $\mathcal P$ is
$k$-branched.

\paragraph{The merge lemmas.}
Throughout this part, $C$ is a connected graph and $\mathcal P$ is a partition of
$V(C)$; we abbreviate $f=\rho_C^{\mathcal P}$ and write $\rho=\rho_C$.  A class
$A\in\mathcal P$ is an \emph{$S$-class} if $\rho(A)=1$ and a \emph{$D$-class} if
$\rho(A)=2$; the \emph{weight} of a subfamily is the sum of the cut-ranks of its
classes.  We call $\mathcal P$ \emph{clean} if every class is an $S$-class or a
$D$-class, the union of any two $S$-classes has cut-rank exactly two, and the union of
any other pair of classes has cut-rank at least three; cleanness is exactly what the
pair-reduction rule below enforces.  A \emph{low triple} is a family of three
$S$-classes whose union has cut-rank at most two.  The five \emph{merge patterns} are
the subfamilies $\mathcal Q\subseteq\mathcal P$ of the following shapes whose union has
cut-rank at most two; here $T,T'$ denote disjoint low triples.  In the connector column,
$X+Y$ means joining the roots of the rooted connectors for $X$ and $Y$ to a new root,
and parentheses specify the order of these joins.
\begin{center}
\begin{tabular}{@{}clll@{}}
\toprule
row & classes of $\mathcal Q$ & weight & connector \\
\midrule
1 & $D,S,S$ & $4$ & $D+(S+S)$ \\
2 & $S,S,S,S$ & $4$ & $(S+S)+(S+S)$ \\
3 & $T\cup\{D\}$ & $5$ & $((S+S)+S)+D$ \\
4 & $T\cup\{S,S\}$ & $5$ & $((S+S)+S)+(S+S)$ \\
5 & $T\cup T'$ & $6$ & $((S+S)+S)+((S+S)+S)$ \\
\bottomrule
\end{tabular}
\end{center}
Every pattern is $2$-branched: the connector in the last column is a rooted tree whose
displayed subfamilies are single classes, pairs of $S$-classes (of cut-rank at most two
by subadditivity), low triples, or the pattern itself.

The engine of the correctness proof is the following consequence of
\zcref{lem:titanic-partition}: it tells us when gluing a subfamily into a single class
is harmless.

\begin{lemma}[Titanic set]\label{lem:titanic-set}
Let $k\geq0$, let $G$ be a graph, and let $\mathcal P$ be a set of pairwise disjoint
nonempty subsets of $V(G)$.  Let $\mathcal Q\subseteq\mathcal P$ be a $k$-branched
subfamily such that for every partition $(\mathcal Q_1,\mathcal Q_2,\mathcal Q_3)$ of
$\mathcal Q$ into three, possibly empty, subfamilies, some $i\in\{1,2,3\}$ satisfies
$\rho_G^{\mathcal P}(\mathcal Q_i)\geq\rho_G^{\mathcal P}(\mathcal Q)$.  Then the
rank-width of $(G,\mathcal P)$ is at most $k$ if and only if the rank-width of
$\bigl(G,(\mathcal P\setminus\mathcal Q)\cup\{\bigcup\mathcal Q\}\bigr)$ is at most~$k$.
\end{lemma}

\begin{proof}
Write $f=\rho_G^{\mathcal P}$ and
$\mathcal P'=(\mathcal P\setminus\mathcal Q)\cup\{\bigcup\mathcal Q\}$.  If
$\abs{\mathcal P'}=1$, then the rank-width of $(G,\mathcal P')$ is zero.  If
$\abs{\mathcal P'}=2$, then its unique cut has value
$f(\mathcal Q)\leq k$, because $\mathcal Q$ is $k$-branched.  In either case, attaching
the at most one class of $\mathcal P\setminus\mathcal Q$ to a rooted tree witnessing
that $\mathcal Q$ is $k$-branched bounds the rank-width of $(G,\mathcal P)$ by~$k$.
Thus assume $\abs{\mathcal P'}\geq3$.

Suppose the rank-width of $(G,\mathcal P')$ is at most $k$.  Since $\mathcal Q$ is
$k$-branched, it carries a rooted subcubic tree whose displayed subfamilies have $f$ at
most $k$.  Replacing the leaf $\bigcup\mathcal Q$ of a witnessing decomposition of
$(G,\mathcal P')$ by that rooted tree gives a rank-decomposition of $(G,\mathcal P)$;
each new edge displays a subfamily of $\mathcal Q$ (of width $\leq k$) and each old edge
keeps its width.

Conversely, suppose the rank-width of $(G,\mathcal P)$ is at most $k$.  The function
$f$ is a symmetric submodular function on the finite ground set $\mathcal P$, and a
rank-decomposition of $(G,\mathcal P)$ of width at most $k$ is exactly a
branch-decomposition of $f$ of width at most $k$; moreover its leaf edges display the
singletons, so $f(\{A\})\leq k$ for every $A\in\mathcal P$.  Consider the partition of
the ground set $\mathcal P$ into $\mathcal Q$ and the singletons $\{A\}$ for
$A\in\mathcal P\setminus\mathcal Q$.  Each singleton is trivially titanic, and
$\mathcal Q$ is titanic by hypothesis; the width of this partition is at most
$\max\bigl(f(\mathcal Q),\max_A f(\{A\})\bigr)\leq k$ because $\mathcal Q$ is
$k$-branched.  By \zcref{lem:titanic-partition}, the quotient function has branch-width
at most $k$, and that quotient function is precisely~$\rho_G^{\mathcal P'}$.
\end{proof}

\begin{lemma}[Pair merges]\label{lem:pair-merge}
Let $C$ be a connected graph and let $\mathcal P$ be a partition of $V(C)$ into classes
of cut-rank at most two.  Let $A,B\in\mathcal P$ be distinct and suppose that either
\begin{enumerate}[label=\rm(\alph*)]
\item $\rho(A\cup B)\leq1$, or
\item $\max\bigl(\rho(A),\rho(B)\bigr)=2$ and $\rho(A\cup B)\leq2$.
\end{enumerate}
Then $\mathcal Q=\{A,B\}$ satisfies the hypothesis of \zcref{lem:titanic-set} with
$k=2$.
\end{lemma}

\begin{proof}
The family $\mathcal Q$ is $2$-branched because $\rho(A\cup B)\leq2$ and its classes
have cut-rank at most two.  A partition of $\mathcal Q$ into three subfamilies either
has $\mathcal Q$ itself as a part, which trivially satisfies the condition, or
separates $A$ from $B$.  In case (a), $f(\{A\})=\rho(A)\geq1\geq\rho(A\cup B)$ because
$C$ is connected and $A$ is a nonempty proper subset of $V(C)$.  In case (b), the
class of cut-rank two satisfies $f\geq2\geq\rho(A\cup B)$.
\end{proof}

For the five patterns we also need a lower bound on the cut-rank of the union of three
$S$-classes.  The following lemma provides it in \emph{every} graph; this is the step
where a primeness assumption would traditionally be invoked.

\begin{lemma}[Rigidity]\label{lem:rigidity}
Let $X_1,X_2,X_3$ be pairwise disjoint nonempty vertex sets in a graph $H$ with
$\rho_H(X_i)\leq1$ for all $i$.  If $\rho_H(X_i\cup X_j)\geq2$ whenever $i\neq j$, then
$\rho_H(X_1\cup X_2\cup X_3)\geq2$.
\end{lemma}
\begin{proof}
  Suppose not. Let $H$ be a counterexample with the minimum number of vertices.
  Let $X_4=V(H)\setminus (X_1\cup X_2\cup X_3)$.
  If $X_4=\varnothing$, then
  $\rho_H(X_1\cup X_2)=\rho_H(X_3)\leq1$, a contradiction.  Thus $X_4$ is nonempty.
  Now $\rho_H(X_1\cup X_2\cup X_3)=\rho_H(X_4)\leq1$.  Moreover,
  $\rho_H(X_i\cup X_4)=\rho_H(X_j\cup X_k)\geq2$ whenever
  $\{i,j,k\}=\{1,2,3\}$, by symmetry of cut-rank.  Hence
  $\rho_H(X_i\cup X_j)\geq2$ for all distinct $i,j\in\{1,2,3,4\}$.
  By the submodularity, for all distinct $i,j\in\{1,2,3,4\}$, 
  $\rho_H(X_i)+\rho_H(X_j)\ge \rho_H(X_i\cup X_j)\ge 2$ and therefore 
  $\rho_H(X_i)=1$ for all $i\in\{1,2,3,4\}$.
  
  For $k\in\{1,2,3,4\}$, minimality implies that no vertex of $X_k$ has
  an empty neighborhood outside~$X_k$ and that no two distinct vertices of~$X_k$
  have the same neighborhood outside $X_k$: in either case, deleting one of the
  vertices would leave all the relevant cut-ranks unchanged.  Since the cut matrix
  between $X_k$ and $V(H)\setminus X_k$ has rank one, all its nonzero rows are equal.
  We deduce that $\abs{X_k}=1$.
  Then $H$ is a prime graph on four vertices, contradicting the fact that every
  connected graph on four vertices has a split, as one checks directly.
\end{proof}

\begin{lemma}[Safety]\label{lem:merge-safe}
Let $C$ be a connected graph and let $\mathcal P$ be a clean partition of $V(C)$.  Let
$\mathcal Q\subseteq\mathcal P$ match one of the five merge patterns; if $\mathcal Q$
matches row 4 or row 5, suppose additionally that no four $S$-classes of $\mathcal P$
have a union of cut-rank at most two.  Then $\mathcal Q$ satisfies the hypothesis of
\zcref{lem:titanic-set} with $k=2$.
\end{lemma}

\begin{proof}
Each pattern is $2$-branched, as observed after the table.  Write
$q=f(\mathcal Q)\leq2$ and consider a partition of $\mathcal Q$ into three subfamilies.
A subfamily equal to $\mathcal Q$ satisfies the condition trivially, and if $q\leq1$
any nonempty subfamily does, because its union is a nonempty proper subset of $V(C)$
and $C$ is connected.  So assume $q=2$; we must find a subfamily of value at least two.
We use, for disjoint subfamilies, both subadditivity
$f(\mathcal X\cup\mathcal Y)\leq f(\mathcal X)+f(\mathcal Y)$ and the reverse bound
$f(\mathcal X\cup\mathcal Y)\geq f(\mathcal X)-f(\mathcal Y)$; both follow from
submodularity and symmetry of the cut-rank function.

First suppose $\mathcal Q$ contains a $D$-class $A$ (rows 1 and 3), and let
$\mathcal A$ be the subfamily containing $A$.  If $\mathcal A=\{A\}$, then
$f(\mathcal A)=2$.  If $\mathcal A$ consists of $A$ and one further class, then
$f(\mathcal A)\geq3$ by cleanness.  If $\mathcal A$ consists of $A$ and two further
classes $B,B'$, then $f(\mathcal A)\geq f(\{A,B\})-f(\{B'\})\geq3-1=2$.  In every case
$f(\mathcal A)\geq q$.

Now suppose all classes of $\mathcal Q$ are $S$-classes (rows 2, 4, and 5, so
$\abs{\mathcal Q}\in\{4,5,6\}$).  Every partition of $\mathcal Q$ into three
subfamilies has a proper subfamily $\mathcal A$ with $\abs{\mathcal A}\geq2$.  If
$\abs{\mathcal A}=2$, then $f(\mathcal A)=2$ by cleanness.  If $\abs{\mathcal A}=3$,
then $f(\mathcal A)\geq2$ by \zcref{lem:rigidity}, whose hypotheses hold because
$S$-classes have cut-rank one and their pairwise unions have cut-rank two by cleanness.
If $\abs{\mathcal A}=4$ (possible only in rows 4 and 5), then $f(\mathcal A)\leq1$
would exhibit four $S$-classes with a union of cut-rank at most two, contradicting the
additional hypothesis; so $f(\mathcal A)\geq2$.  If $\abs{\mathcal A}=5$ (row 5 only),
then $f(\mathcal A)\leq1$ would give every four of its classes a union of cut-rank at
most $f(\mathcal A)+1\leq2$ by subadditivity, again contradicting the
additional hypothesis.  In every case some subfamily has value at least $q=2$.
\end{proof}

\begin{remark}\label{rem:not-every-merge}
The additional hypothesis of \zcref{lem:merge-safe} for rows 4 and 5 cannot be
dropped.  Let $W$ be a set of four vertices
inducing a path whose vertices all have the same nonempty neighborhood outside~$W$ (one
side of a split).  Then $\rho(W)=1$ while every pair and every triple inside $W$ has
cut-rank two, so for any vertex $x\notin W$ the five singletons of $W\cup\{x\}$ match
the shape of row 4 and $\rho(W\cup\{x\})\leq2$.  This family is not titanic: the
partition $(W,\{x\},\varnothing)$ has $\rho(W)=\rho(\{x\})=1$.  The danger is real: in
a systematic experiment on 13-vertex graphs of this kind, merging $W\cup\{x\}$
destroyed width-two decomposability in $933{,}552$ of $1{,}700{,}988$ generated
configurations.  The exhaustive comparison on all graphs with at most ten vertices found no
mismatch, but that observation alone could not establish safety: a bad merge can affect
the returned answer only if later reductions leave a stuck state with more than six
representatives.  The 13-vertex experiment supplies explicit false-negative states when
the unsafe merge is forced.  The algorithm never performs that merge: $W$
itself matches row 2; that merge
\emph{is} safe; and rows 4 and 5 are attempted only after row 2 is
exhausted---precisely the additional hypothesis of \zcref{lem:merge-safe}.
\end{remark}

Two more lemmas complete the toolkit: a stopping rule and a guarantee of progress.

\begin{lemma}[Base case]\label{lem:six-reps}
Let $C$ be a connected graph and let $\mathcal P$ be a partition of $V(C)$ into classes
of cut-rank at most two.  If the weight of $\mathcal P$ is at most six, then the
rank-width of $(C,\mathcal P)$ is at most two.
\end{lemma}

\begin{proof}
If $\abs{\mathcal P}\leq2$ the rank-width is at most two, so assume $\abs{\mathcal P}\geq3$;
then every class is a nonempty proper subset of $V(C)$ and has cut-rank one or two.
Partition $\mathcal P$ into three nonempty groups of weight at most two each: start with
three empty groups, place each $D$-class in a group of its own---there are at most
three---and distribute the $S$-classes among the remaining capacity, first making every
unused group nonempty.  This is possible because $|\mathcal P|\geq3$ and the total
weight is at most six.  Each group is a single class or a pair of
$S$-classes, so the cut-rank of its union is at most two by subadditivity.  Join the
three groups at one internal node of a tree, splitting each two-class group by one
further node.  Every displayed subfamily is a group, a single class, or the complement
of a group, hence has cut-rank at most two.
\end{proof}

\begin{lemma}[Side weight]\label{lem:side-weight}
Let $T$ be a subcubic tree whose leaves carry weights in $\{1,2\}$ with total weight at
least seven.  Then some edge of $T$ splits off a component whose leaf weights sum to
$4$, $5$, or~$6$.
\end{lemma}

\begin{proof}
Call a component of $T\setminus e$, for an edge $e$ of $T$, a \emph{side}.  Sides of weight at
least four exist: the side complementary to any leaf has weight at least $7-2=5$.
Choose a side $Y$ of weight at least four with minimum weight and suppose its weight is
at least seven.  Then $Y$ is not a single leaf, so the end of its defining edge inside
$Y$ is an internal node, whose two branches are sides partitioning the weight of $Y$;
the heavier branch has weight at least $\lceil7/2\rceil=4$ and less than the weight of
$Y$, contradicting minimality.
\end{proof}

\begin{lemma}[Progress]\label{lem:merge-exists}
Let $C$ be a connected graph and let $\mathcal P$ be a clean partition of $V(C)$ of
weight at least seven.  If the rank-width of $(C,\mathcal P)$ is at most two, then
$\mathcal P$ contains an instance of one of the five merge patterns.
\end{lemma}

\begin{proof}
Fix a rank-decomposition $(T,\mu)$ of $(C,\mathcal P)$ of width at most two and give
the leaf of each class $A$ the weight $\rho(A)$, so that the total weight is at least
seven.  Every side of $T$ displays a subfamily whose union has cut-rank at most two.
By \zcref{lem:side-weight} some side $\mathcal C$ has weight $w\in\{4,5,6\}$.

First suppose $\mathcal C$ contains a $D$-class.  Root $\mathcal C$ at its defining
edge, choose a $D$-class $A$ in $\mathcal C$ whose leaf is deepest, and let
$\mathcal Z$ be the sibling branch of that leaf, so that $\mathcal Z$ and
$\{A\}\cup\mathcal Z$ are displayed subfamilies inside~$\mathcal C$ and
$f(\{A\}\cup\mathcal Z)\leq2$.  If $\abs{\mathcal Z}=1$, then $\{A\}\cup\mathcal Z$ is
a pair of classes involving a $D$-class whose union has cut-rank at most two,
contradicting cleanness; so $\abs{\mathcal Z}\geq2$, all classes of~$\mathcal Z$ lie
strictly deeper than $A$ and are therefore $S$-classes, and
$\abs{\mathcal Z}\leq w-2\leq4$.  If $\abs{\mathcal Z}=2$, row 1 succeeds on
$\{A\}\cup\mathcal Z$.  If $\abs{\mathcal Z}=3$, then $\mathcal Z$ is a displayed low
triple and row 3 succeeds on $\mathcal Z\cup\{A\}$.  If $\abs{\mathcal Z}=4$, then
$\mathcal Z$ is a displayed family of four $S$-classes of cut-rank at most two and row
2 succeeds.

Now suppose all classes of $\mathcal C$ are $S$-classes, so $\abs{\mathcal C}=w$.  If
$w=4$, row 2 succeeds on $\mathcal C$.  If $w=5$, the top node of $\mathcal C$ splits
it into displayed subfamilies of sizes $\{1,4\}$ or $\{2,3\}$.  In the case $\{2,3\}$,
the triple is a displayed low triple $T'$ and row 4 succeeds on $T'$ together with the
remaining two classes, because $f(\mathcal C)\leq2$.  In the case $\{1,4\}$, the
four-class side either splits $\{1,3\}$, yielding a displayed low triple $T'$ and row 4
on $T'$ plus the remaining two classes of $\mathcal C$, or splits $\{2,2\}$, in which
case that four-class side itself yields row 2.  If $w=6$, the top node splits
$\mathcal C$ into sizes $\{3,3\}$, $\{2,4\}$, or $\{1,5\}$.  For $\{3,3\}$, both halves
are displayed low triples and row 5 succeeds.  For $\{2,4\}$, the four-class side
yields row 2.  For $\{1,5\}$, apply the $w=5$ analysis to the five-class side $Y$,
noting $f(Y)\leq2$: it yields row 2, or a displayed low triple $T'\subseteq Y$ giving
row 4 on $T'$ plus the other two classes of $Y$.
\end{proof}

\paragraph{The reduction algorithm.}
We now turn the preceding lemmas into a self-contained procedure for deciding whether a
graph has rank-width at most two.  Its primary state is a partition; the blocks carry
partial decompositions only as certificates that the cuts already constructed have
width at most two.

Fix a connected component $C$.  A class $B$ of the current partition, which we call a
\emph{block}, comes with a rooted partial rank-decomposition on $B$ whose displayed
cuts, evaluated in $C$, have cut-rank at most two; as for classes, a block is an
\emph{$S$-block} or a \emph{$D$-block} according to whether its cut-rank is one or two.
It also comes with a set
$P(B)\subseteq B$ of at most two representative vertices whose rows form a basis of
$A(C)[B,V(C)\setminus B]$.  
If~$\mathcal B$ is the set of current blocks, put
\[
  R=\bigcup_{B\in\mathcal B}P(B)
  \quad\text{and}\quad
  P(\mathcal A)=\bigcup_{B\in\mathcal A}P(B)
  \quad(\mathcal A\subseteq\mathcal B).
\]
The representatives retain exactly the information needed for every subsequent
cut-rank test:
\begin{equation}\label{eq:representative-rank}
 \rho_C\left(\bigcup\mathcal A\right)
 =\rank_{\Ftwo} A(C)
   [P(\mathcal A),R\setminus P(\mathcal A)].
\end{equation}
Indeed, the representatives span the rows on one side of the cut and, by symmetry of the
adjacency matrix, the representatives outside $\mathcal A$ span its columns.  In
particular, $\abs{P(B)}=\rho_C(B)$ for every proper block $B$, and $\abs R$ is the
weight of the current partition.

Suppose that a family $\mathcal Q$ of blocks has union of cut-rank at most two.  Row
reduction of the matrix in \eqref{eq:representative-rank} selects at most two vertices
of $P(\mathcal Q)$ whose rows form a basis.  These vertices become the representatives
of the new block $\bigcup\mathcal Q$.  To \emph{merge} $\mathcal Q$, we also join its
rooted partial decompositions by the connector prescribed in the pattern table (or by a
single binary join when $\mathcal Q$ consists of two blocks).  Thus every displayed cut
created inside the new block has cut-rank at most two, and
\eqref{eq:representative-rank} remains valid after the merge.

There are two kinds of admissible merges.  A \emph{pair reduction} merges two distinct
blocks $B_1,B_2$ if either both are $S$-blocks and
$\rho_C(B_1\cup B_2)\leq1$, or at least one is a $D$-block and
$\rho_C(B_1\cup B_2)\leq2$.  By \zcref{lem:pair-merge}, this reduction is safe.
Moreover, no pair reduction is available exactly when the current partition is clean.
A \emph{pattern reduction} merges a family that occurs in one of the five rows of the
pattern table.  Its position in the ordered list below is important: rows 4 and 5 are
used only after row 2 has been exhausted, as required by \zcref{lem:merge-safe}.

\begin{algorithm}[H]
\caption{Deciding whether a graph has rank-width at most two}
\label{alg:rw2}
\begin{algorithmic}[1]
\Require A graph $G$
\ForAll{connected components $C$ of $G$ with $\abs{V(C)}\geq7$}
  \State Set $\mathcal B=\{\{v\}:v\in V(C)\}$ and $P(\{v\})=\{v\}$ for every
         $v\in V(C)$; thus $R=V(C)$.
  \State Apply the first applicable rule in the following ordered list.
  \Statex \hspace{\algorithmicindent}\textup{(a)}
    If $\abs R\leq6$, then declare the reduction of $C$ successful.
  \Statex \hspace{\algorithmicindent}\textup{(b)}
    If a pair reduction is available, then perform one.
  \Statex \hspace{\algorithmicindent}\textup{(c)}
    If no pair reduction is available and a pattern from row 1 or row 2 is present,
    then merge one such pattern.
  \Statex \hspace{\algorithmicindent}\textup{(d)}
    If none of (a)--(c) applies and a pattern from rows 3--5 is present,
    then merge one such pattern.
  \Statex \hspace{\algorithmicindent}\textup{(e)}
    If $\abs R>6$ and none of (b)--(d) applies, then declare the reduction of $C$
    unsuccessful.
  \State After a merge in (b), (c), or (d), inspect the ordered list again from (a).
\EndFor
\State The answer is \textbf{yes} if every component considered above is declared
       successful, and \textbf{no} if one is declared unsuccessful.
\end{algorithmic}
\end{algorithm}

Components with at most six vertices need not be listed in
\zcref{alg:rw2}: their singleton partitions satisfy \zcref{lem:six-reps}.  In each of
(b)--(d), any applicable family may be chosen; no backtracking is required.  Rule (b)
is found by examining all pairs.  For rule (c), one examines every $D$-block with every
pair of $S$-blocks and every set of four $S$-blocks.  Rule (d) can be carried out by
first listing all low triples $T$.  For each $T$, one tests
$T\cup\{D\}$ for every $D$-block, $T\cup\{S,S\}$ for every two further $S$-blocks,
and $T\cup T'$ for every disjoint low triple~$T'$.  These are precisely the families in
rows 3--5.

The same formulation explains why the procedure is fast.  If there are $b$ blocks,
one pass through (b) tests $O(b^2)$ pairs, (c) tests $O(b^4)$ families, and (d) tests
$O(b^6)$ families.  Every merge decreases the number of blocks, so fewer than $n$
passes can perform a merge.  The crude bound is therefore $O(n^7)$ cut-rank tests.
Each such test is Gaussian elimination on the representative matrix in
\eqref{eq:representative-rank}; for the graphs considered here this matrix is encoded
by a handful of binary words.

\begin{proposition}\label{prop:rw2-algorithm}
\zcref{alg:rw2} answers \textbf{yes} if and only if $G$ has rank-width at most two.
\end{proposition}

\begin{proof}
Rank-width is the maximum over connected components, and components with at most six
vertices are settled by \zcref{lem:six-reps} applied to singleton classes.  So
fix a component $C$ with at least seven vertices and consider the invariant
\[
 \tag{$\ast$} \label{eq:ast}
 \text{the rank-width of $(C,\mathcal B)$ is at most two if and only if
 $\rw(C)\leq2$,}
\]
which holds initially because the blocks are singletons.  A merge in rule (b) satisfies
\zcref{lem:pair-merge}.  When rule (c) applies, no pair reduction is available, so the
partition is clean and \zcref{lem:merge-safe} applies.  When rule (d) applies, the
partition is clean and row 2 is unavailable; hence no four $S$-classes have a union of
cut-rank at most two, which is the additional hypothesis of
\zcref{lem:merge-safe} for rows 4 and 5.  Thus every family merged in (b)--(d) satisfies
the hypothesis of \zcref{lem:titanic-set} with $k=2$.  Each merge preserves \eqref{eq:ast},
and the representative construction above preserves \eqref{eq:representative-rank}.

If rule (a) applies, then the rank-width of $(C,\mathcal B)$ is at most two by
\zcref{lem:six-reps}, and \eqref{eq:ast} gives $\rw(C)\leq2$.  If rule (e) applies, the
partition is clean, has weight at least seven, and contains none of the five merge
patterns.  The contrapositive of \zcref{lem:merge-exists} gives that the rank-width of
$(C,\mathcal B)$ exceeds two, and \eqref{eq:ast} gives $\rw(C)>2$.  Every merge strictly
decreases the number of blocks, so one of these two conclusions is reached after
finitely many steps.  The final answer is correct because rank-width is the maximum
over the connected components.
\end{proof}

An independent verifier specifies the same decision problem by the following direct
subset dynamic program from Oum~\cite{Oum2009}.  For $S\subseteq V(G)$, let $D_k(S)$ hold when $S$ can
be the leaf set of a rooted binary partial rank-decomposition whose displayed cuts all
have cut-rank at most $k$.  Then
\begin{equation}\label{eq:subset-dp}
 D_k(S)=\bigl[\rho_G(S)\leq k\bigr]\ \wedge\
 \left(
 |S|=1\ \vee\
 \bigvee_{\substack{\varnothing\neq A\subsetneq S}}
       \bigl(D_k(A)\wedge D_k(S\setminus A)\bigr)
 \right).
\end{equation}
Thus $\rw(G)\leq k$ if and only if $D_k(V(G))$ holds.  Symmetry breaking by fixing one
vertex in $A$ gives an $O(3^n)$ implementation, which is more than adequate for the 25
graphs with at most ten vertices and their reductions.  This verifier is independent of the
optimized implementation used in the search.  It is, however, no substitute for
it: on 16-vertex inputs the implementation of \zcref{alg:rw2} decides a graph in roughly a
microsecond, while an optimized implementation of \eqref{eq:subset-dp} needs seven to
eight milliseconds on a laptop---a factor of several thousand that the $9\times10^{10}$ rank-width
decisions of the final layer cannot absorb: at that rate the final layer alone would
require about 20 CPU-years.

By \zcref{prop:three-reductions,prop:monotone}, a graph $G$ of rank-width greater than two
is an excluded vertex-minor if and only if, for every $v\in V(G)$, the three graphs
\[
  G\setminus v,\qquad G*v\setminus v,\qquad G/v
\]
have rank-width at most two.  This is the test used for vertex-minor candidates.  For a
pivot-minor candidate the analogous test checks ordinary deletion and pivot-deletion at
every vertex.  By \zcref{prop:two-pivot-reductions}, these are all one-vertex
pivot-minors up to pivot equivalence.

\subsection{Further optimizations}\label{sec:further-optimizations}

Three additional refinements reduce the number of expensive rank-width and
classification calls.  First, before classifying an extension $G$, we test the deletion
of two fixed old vertices.  In the implementation's zero-based labeling, the old
vertices of an $n$-vertex child are $0,\ldots,n-2$, the new vertex is $n-1$, and the
two selected old vertices are $0$ and $n-2$.  If one of these
deletions has rank-width greater than two, then $G$ is neither a member of the next
frontier nor minimal under vertex-minors, and no further work on $G$ is needed.  The
deletion of the newly added vertex is not tested, since it is the frontier parent and is
already known to have rank-width at most two.

Second, the 16-vertex layer is treated separately because rank-width-two children will
not be extended.  We test the old-vertex deletions before testing $G$ itself and classify
only candidates that survive the complete minimality test.  For a 16-vertex child,
the deletion of old vertex 14 is
tested first.  Extensions that differ only in their adjacency to vertex 14 give the same
graph after this deletion, so its rank-width result can be reused for the paired
extensions.

Finally, for pivot-minor candidates on 15 or 16 vertices, every excluded pivot-minor must
have a four-vertex set $X$ with $\rho_G(X)\leq2$, by
\zcref{prop:excluded-connectivity,prop:excluded-not-three-plus-one}.  Such a set in a
child either consists of four old vertices or of the new vertex and three old vertices.
We therefore record the relevant low-rank four-sets and three-sets of each parent and
retest only the inherited possibilities in its children.  This avoids an exhaustive
four-set scan for every extension without changing the set of candidates.

\section{Exact local-equivalence classification}\label{sec:lc-key}

Bouchet gave a polynomial-time algorithm for recognizing whether two graphs are locally
equivalent~\cite{Bouchet1991},
but in our case, we need to identify locally equivalent graphs up to isomorphisms.
The exhaustive search additionally needs a compact
canonical key for grouping millions of graphs, while traversing the whole
local-complementation orbit of every candidate is already too slow for graphs on 12
vertices.  Our
classifier uses an exact invariant of Bouchet's isotropic system; the terminology of
graph-state stabilizers is only a second interpretation of the same binary space.  We
give the complete correspondence before describing the implementation.

\subsection{From a graph to an isotropic system}

Let $K=\Ftwo^2=\{0,X,Z,Y\}$, where
\[
  X=(1,0),\qquad Z=(0,1),\qquad Y=(1,1),
\]
and equip $K$ with the alternating form
\[
  \langle(x,z),(x',z')\rangle_K=xz'+zx'.
\]
Thus two distinct nonzero elements of $K$ have inner product one.
For a vertex set $V$ and two vectors $x,y\in K^V$,
we define $\langle x,y\rangle= \sum_{v\in V} \langle x(v),y(v)\rangle_K$.
A subspace $L$ of $K^V$ is called \emph{totally isotropic} if $\langle x,y\rangle=0$ for all $x,y\in L$.

An \emph{isotropic system}
on $V$ is a $|V|$-dimensional totally isotropic subspace of $K^V$; this is Bouchet's
original language \cite{Bouchet1987a}.
(Bouchet used $\alpha$, $\beta$, $\gamma$ instead of $X$, $Y$, $Z$, respectively.)

For a graph $G$ with $V(G)=[n]$, define
\begin{equation}\label{eq:code}
  L(G)=\{(x,xA(G)):x\in\Ftwo^n\}
      =\operatorname{rowspan}_{\Ftwo}[I_n\mid A(G)]\subseteq K^{[n]}.
\end{equation}
The $i$th pair of coordinates is read as an element of $K$.  The identity block gives
$\dim L(G)=n$, and symmetry of $A(G)$ gives
\[
 \langle(x,xA(G)),(y,yA(G))\rangle
 =xA(G)y^{\mathsf T}+xA(G)^{\mathsf T}y^{\mathsf T}=0.
\]
Hence $L(G)$ is an isotropic system.  In Bouchet's terminology, the graph and the two
constant vectors $X,Z$ form a graphic presentation of $L(G)$
\cite{Bouchet1988}.  In the graph-state terminology, the rows of
$[I_n\mid A(G)]$ are stabilizer generators; no phase information is used here.

Every permutation of the three nonzero elements $X,Z,Y$ preserves
$\langle\cdot,\cdot\rangle_K$.  A \emph{local symbol map} is an independent such
permutation at each vertex.  Together with a permutation of the vertices, these maps
form the group
\[
  \Gamma_n=S_n\ltimes S_3^n
\]
acting linearly on $K^{[n]}$.  The following theorem is the precise bridge to local
complementation.

\begin{theorem}[Bouchet \cite{Bouchet1988}; Van den Nest, Dehaene, and De Moor~\cite{NDM2004}]
\label{thm:isotropic-lc}
Two $n$-vertex graphs $G$ and $H$ are locally equivalent up to isomorphism if and only
if some element of $\Gamma_n$ maps $L(G)$ onto $L(H)$.
\end{theorem}

\begin{proof}[References and explanation]
Changing the basis of the $n$ displayed generators does not change $L(G)$.  In the
binary stabilizer notation, Eq.~(1) of Van den Nest, Dehaene, and De Moor
\cite{NDM2004} says that two stabilizer subspaces are locally Clifford equivalent
exactly when one is obtained from the other by an invertible change of generators and
an invertible $2\times2$ binary map at each coordinate.  The six invertible
$2\times2$ matrices over~$\Ftwo$ are precisely the six permutations of $X,Z,Y$.
Their Theorem~2 identifies local complementation at a vertex with one such local map,
and their Theorem~3 proves that these complementations generate every local Clifford
transformation between graph presentations.  This is also the content of the theory of
graphic presentations of isotropic systems developed by Bouchet \cite{Bouchet1988};
including $S_n$ accounts for graph isomorphism because relabelings of $G$ correspond
exactly to coordinate permutations of $K^{[n]}$.
\end{proof}

Thus the classification problem has become a finite linear-algebra problem: canonize
the subspace $L(G)$ under coordinate permutations and independent permutations of the
three nonzero symbols.

\subsection{Canonical encoding of a graph up to local equivalence and isomorphism}

The encoding used here is adapted from the code graph of Danielsen and
Parker~\cite{DP2009}, which in turn extends Östergård's method for classifying
linear codes~\cite{Ostergard2002}.  Their construction represents selected
low-weight words of an additive code over $\operatorname{GF}(4)$ by a colored incidence graph
having three nonzero-symbol vertices at each coordinate, and canonizes that
graph with nauty.  We apply the same basic construction to the isotropic space
$L(G)$.  The additional safeguard used below is to test the automorphisms of
the truncated incidence graph on the full space and to increase the weight
cutoff when that test fails.

For $a\in K^{[n]}$, its \emph{support} is
$\operatorname{supp}(a)=\{i:a(i)\neq0\}$ and its weight is the size of its support.  Weight is invariant under~$\Gamma_n$.  For a cutoff $t$, let
\[
  W_t(G)=\{a\in L(G)\setminus\{0\}:|\operatorname{supp}(a)|\leq t\}.
\]
Construct a colored incidence graph $\mathcal I_t(G)$ with three kinds of vertices:
\begin{enumerate}
  \item a coordinate vertex $q_i$ for each $i\in[n]$;
  \item three symbol vertices $X_i,Z_i,Y_i$, each adjacent to $q_i$;
  \item a word vertex $w_a$ for each $a\in W_t(G)$, adjacent to the single symbol
        vertex $a(i)_i$ for every $i\in\operatorname{supp}(a)$.
\end{enumerate}
The three kinds are separate color classes.  The stars with center $q_i$ and leaves
$X_i,Z_i,Y_i$ force every color-preserving isomorphism to map an entire
coordinate, together with its three symbols, to another coordinate.

\begin{lemma}\label{lem:incidence-action}
A color-preserving isomorphism from $\mathcal I_t(G)$ to $\mathcal I_t(H)$ is the same
thing as an element $\gamma\in\Gamma_n$ satisfying
$\gamma(W_t(G))=W_t(H)$.
\end{lemma}

\begin{proof}
An isomorphism first gives a bijection of the coordinate vertices.  The three symbol
neighbors of each coordinate must be carried to the three symbol neighbors of its
image, giving one permutation of $X,Z,Y$ at that coordinate.  The neighborhood of a
word vertex records every nonzero entry of the word, so its image is forced to be the
word vertex for the transformed word.  This proves one direction, and the reverse
construction is immediate.
\end{proof}

Taking $t=n$ therefore gives an exact but potentially large representation: it has one
word vertex for each of the $2^n-1$ nonzero vectors of $L(G)$.  The implementation
usually obtains the same canonical coordinate and symbol orders from only the
low-weight words.  Sparse nauty \cite{MP2014} canonically labels
$\mathcal I_t(G)$ and supplies generators of its color-preserving automorphism group.
Each generator is converted, using \zcref{lem:incidence-action}, to an element of
$\Gamma_n$ and tested on the \emph{full} space $L(G)$.  It is enough to transform the
$n$ rows of $[I_n\mid A(G)]$ and test whether their images lie in $L(G)$, because both
the action and the space are linear.  If a generator fails, the cutoff is increased
and the incidence graph is canonized again.  If all generators pass, every
automorphism of~$\mathcal I_t(G)$ preserves $L(G)$.  At~$t=n$ this condition is
automatic, so the procedure always terminates.

The accepted canonical labeling orders the $n$ coordinates and the three symbols
inside each coordinate.  Apply this column transformation to the original $n$ rows of
$[I_n\mid A(G)]$ and take binary reduced row-echelon form (RREF).
RREF removes the arbitrary
choice of a generator basis.
The whole construction can be written
as follows.  Its number of rows and row length determine $n$, so keys from graphs with
different numbers of vertices cannot collide.

\begin{algorithm}[H]
\caption{The exact adaptive local-equivalence key}
\label{alg:lc-key}
\begin{algorithmic}[1]
\Require An $n$-vertex graph $G$
\Ensure A canonical key for local equivalence and isomorphism
\State enumerate all $2^n-1$ nonzero words of $L(G)$ and their weights
\State Let $t$ run successively from
       $\min\{n,\max\{5,\text{minimum nonzero weight}\}\}$ to $n$
\State For each such $t$, canonically label $\mathcal I_t(G)$ and obtain its
       automorphism generators
\State Choose the first $t$ for which every generator preserves the full space $L(G)$
\State transform $[I_n\mid A(G)]$ to the canonical coordinate and symbol orders
\State \Return the binary RREF of the transformed rows
\end{algorithmic}
\end{algorithm}

The lower cutoff 5 is an empirical tuning choice: among representative frontier
samples it minimized classification time.  It has no role in correctness because a
failed full-space automorphism check raises the cutoff, eventually to $n$.

\begin{lemma}\label{lem:key-exact}
Two graphs receive the same key if and only if they are locally equivalent up to graph
isomorphism.
\end{lemma}

\begin{proof}
Suppose first that $G$ and $H$ are equivalent.  By \zcref{thm:isotropic-lc}, an element
of $\Gamma_n$ maps $L(G)$ to~$L(H)$ and, because it preserves weight, maps
$W_t(G)$ to $W_t(H)$ for every $t$.  The two adaptive runs therefore use isomorphic
incidence graphs and accept at the same cutoff.  Two canonical labelings of the same
colored incidence graph can differ only by one of its automorphisms.  At the accepted
cutoff every such automorphism preserves the full isotropic space.  Conjugating by an
equivalence transports the automorphism group and the full-space acceptance condition,
so the two runs pass or fail together at each cutoff.  The transformed
full spaces and their unique RREFs agree.

Conversely, equality of keys says that the two canonical column transformations send
$L(G)$ and~$L(H)$ to the same subspace.  Reversing one transformation and applying the
other gives an element of $\Gamma_n$ mapping~$L(G)$ onto $L(H)$.
\zcref{thm:isotropic-lc} then gives local equivalence up to isomorphism.  Notice that the
key contains the transformed full space, not merely its low-weight subset; the adaptive
step therefore cannot create a false equality.
\end{proof}

\subsection{Canonical encoding of a graph up to pivot equivalence and isomorphism}

For $S\subseteq V(G)$, let $\sigma_S$ interchange $X_i$ and $Z_i$ for every $i\in S$
and fix $Y_i$ and all symbols at coordinates outside~$S$.  These are precisely the
local symbol maps preserving the unordered pair $\{X_i,Z_i\}$ at every coordinate.
The following matrix interpretation identifies this subgroup with graph pivots; compare~\cite[Propositions~2.5--2.6]{Oum2004}.

\begin{proposition}\label{prop:pivot-code}
Two graphs $G$ and $H$ are pivot-equivalent up to graph isomorphism if and only if a
coordinate permutation followed by some $\sigma_S$ maps $L(G)$ to $L(H)$.
\end{proposition}

\begin{proof}
After applying the coordinate permutation, suppose that
$\sigma_S(L(G))=L(H)$.  Then $L(H)$ has graphic presentations with underlying
graphs $G$ and $H$, respectively, whose presentation vectors take values in
$\{X,Z\}$.  Hence $G$ and $H$ are pivot-equivalent by
Oum~\cite[Proposition~10.1]{Oum2004a}.

Conversely, pivoting an edge interchanges the two presentation vectors at its
ends \cite[(8.3)]{Bouchet1988}.  Thus every sequence of pivots acts on $L(G)$
by some $\sigma_S$.  A coordinate permutation accounts for graph isomorphism.
\end{proof}

The pivot classifier refines the incidence-graph coloring by putting $X_i,Z_i$ in one
color class and~$Y_i$ in another.  A color-preserving isomorphism may independently
exchange $X_i$ and $Z_i$ but cannot send either to $Y_i$; hence its local action is
exactly some $\sigma_S$.  The adaptive automorphism check and full-space row reduction
are otherwise unchanged.  Replacing $\Gamma_n$ by this subgroup in the proof of
\zcref{lem:key-exact} therefore gives an exact pivot-equivalence and isomorphism key.

\section{Completeness and correctness}\label{sec:correctness}

\begin{proposition}\label{prop:frontier-induction}
For every $5\leq n\leq15$, the computed set $\cF_n$ contains exactly one representative
of each local-equivalence and isomorphism class of prime $n$-vertex graphs of rank-width
at most two.
\end{proposition}

\begin{proof}
The assertion holds at $n=5$: the unique relevant class is represented by $C_5$.
Assume it holds for graphs on $n$ vertices, and let $G$ be a prime $(n+1)$-vertex graph
of rank-width at most two.  By \zcref{thm:bouchet-chain}, $G$ has a prime
vertex-minor $H$ on $n$ vertices.  By monotonicity, $\rw(H)\leq2$, so the class of $H$
occurs in $\cF_n$.

By \zcref{prop:three-reductions}, there is a graph $\widehat G$ locally equivalent to
$G$ and a vertex $v$ such that $\widehat G\setminus v$ is locally equivalent to $H$,
up to relabeling.  The local complementations that take $\widehat G\setminus v$ to the
chosen representative act only at vertices other than $v$, so they lift to
$\widehat G$.  Hence a one-vertex extension of the chosen representative lies in the
local-equivalence class of $G$.  By
\zcref{lem:prime-extension}, its new vertex has degree at least two and no twin, so that
extension is enumerated.  The exact rank-width test retains it, and
\zcref{lem:key-exact} places it in exactly one output class.  Conversely, every graph
retained by the algorithm is explicitly verified to be prime and to have rank-width at
most two.  This proves the induction.
\end{proof}

\begin{proposition}\label{prop:obstruction-completeness}
For each $6\leq n\leq16$, the computed set $\cO_n$ contains exactly one representative
of each excluded-vertex-minor class on $n$ vertices.
\end{proposition}

\begin{proof}
Let $G$ be an excluded vertex-minor with $n$ vertices.  Since $\rw(G)=3$ and both
$G\setminus v$ and $G/v$ have rank-width at most two for every vertex $v$,
\zcref{prop:primesubgraph}
implies that $G$ is prime.  The prime
chain theorem \zcref{thm:bouchet-chain} gives a prime vertex-minor
$H$ on $n-1$ vertices, and minimality gives $\rw(H)\leq2$.  By
\zcref{prop:frontier-induction}, the parent class of $H$ is in
$\cF_{n-1}$; the corresponding one-vertex extension of $H$ is enumerated.  
Its full rank-width and elementary-reduction tests accept it, and the exact key
retains one class representative.  Conversely, every output graph is checked to have
rank-width greater than two and all three elementary reductions of rank-width at most
two, which is sufficient by \zcref{prop:three-reductions,prop:monotone}.
\end{proof}

\begin{proof}[Proof of \zcref{thm:main}]
Graphs on at most five vertices have rank-width at most two.  By
\zcref{prop:obstruction-completeness}, the search finds every excluded vertex-minor with
at most 16 vertices.  The computed sets have sizes $1,18,6$ for graphs on 8, 9, and 10
vertices, respectively, and are empty for graphs on 6 or 7 vertices and for graphs with
between 11 and 16 vertices.  Every proper pivot-minor of an excluded vertex-minor is a
proper vertex-minor, so it has rank-width at most two; the graph itself has
rank-width three.  Thus it satisfies the hypothesis of
\zcref{thm:excluded-pivot-bound}, and no excluded vertex-minor has more than~16 vertices.
Exact key comparison shows that the 25 outputs represent distinct local-equivalence and
isomorphism classes.  The elementary-reduction test certifies their minimality.
\end{proof}

Every excluded
vertex-minor is an excluded pivot-minor.  
We remark that out of the 609
excluded pivot-minor classes, 
480 of them were found from 
the 25 excluded vertex-minor classes.
The exact splitter decomposes the 25
local-equivalence classes into $2+414+64=480$ pivot-equivalence classes on 8, 9, and
10 vertices, respectively.

\subsection{Completeness of the pivot-minor search}

For $5\leq n\leq12$, let $\cPiv_n$ contain the computed representatives of the
pivot-equivalence and isomorphism classes of prime $n$-vertex graphs of rank-width at
most two.  The pivot run constructs these complete frontiers inductively.  No complete
pivot frontier is constructed for graphs on more than 12 vertices: for candidate graphs
on 13, 14, 15, or 16 vertices, \zcref{thm:oum-chain} restricts the possible parents to
prime $3^{+3}$-rank-connected graphs.  The search obtains those parents by filtering
$\cPiv_{12}$ for the 13-vertex search and by filtering and splitting the complete
local-equivalence frontiers $\cF_{13},\ldots,\cF_{15}$ for the later searches; see the
proof of \zcref{prop:pivot-completeness}.
The pivot run also
reconstructs the 12-vertex frontier by a second route: because pivot equivalence refines
local equivalence on a fixed vertex set, each complete local-equivalence class from
\zcref{prop:frontier-induction} can be split into pivot-equivalence classes.  For each
stored representative, we enumerate its complete labeled local-complementation
orbit, quotient by ordinary graph isomorphism, and deduplicate with the
exact pivot-equivalence key.

\begin{proposition}\label{prop:pivot-frontier}
For every $5\leq n\leq12$, the computed set $\cPiv_n$ contains exactly one representative
of each pivot-equivalence and isomorphism class of prime $n$-vertex graphs of rank-width
at most two.
\end{proposition}

\begin{proof}
For graphs on five vertices, splitting Bouchet's unique local-equivalence class in this
way gives the single pivot-equivalence class represented by $C_5$.  Assume the assertion
for $\cPiv_n$, and let $G$ be a
prime $(n+1)$-vertex graph of rank-width at most two.  By \zcref{thm:allys-chain}, either
$G$ has a prime pivot-minor $H$ on $n$ vertices or $G$ is pivot-equivalent to the cycle
$C_{n+1}$.  In the first case, monotonicity gives $\rw(H)\leq2$, so the class of $H$
occurs in $\cPiv_n$.  Choose a graph $\widehat G$ pivot-equivalent to $G$ such that
$\widehat G\setminus v$ is isomorphic to $H$.  Every pivot within
$\widehat G\setminus v$ also applies in $\widehat G$, so the pivots and relabeling that
take $H$ to its stored representative lift, and a one-vertex extension reaches the class
of $G$.  In the second case, the algorithm explicitly inserts $C_{n+1}$, which
is prime and has rank-width at most two.  The exact pivot-equivalence key retains one representative
of each class, and every retained graph is directly checked to be prime and of rank-width
at most two.  This proves the induction.
\end{proof}

As a cross-check by the second construction, splitting all $7{,}712$
local-equivalence classes of graphs on 12 vertices directly gives the same
$1{,}053{,}415$ pivot-equivalence classes as the incremental construction,
with identical exact-key union.
For the 13-vertex search, the computation reads the $263{,}343$ classes obtained by
retaining from $\cPiv_{12}$ only the $3^{+3}$-rank-connected graphs.  For searches on
14-, 15-, and 16-vertex graphs, the parent sets are instead obtained by
selecting the relevant classes from the complete local-equivalence frontier and splitting
them into pivot-equivalence classes.

\begin{proposition}\label{prop:pivot-completeness}
For each $6\leq n\leq16$, the pivot-minor search contains exactly one representative of
each excluded pivot-minor class on $n$ vertices.
\end{proposition}

\begin{proof}
Let $G$ be an excluded pivot-minor with $n$ vertices.  It has rank-width three, and
\zcref{prop:excluded-connectivity} implies that it is prime and
$3^{+2}$-rank-connected.

Suppose first that $6\leq n\leq12$.  By \zcref{thm:allys-chain}, $G$ has a prime
pivot-minor $H$ with $n-1$ vertices unless
it is pivot-equivalent to a cycle.  The exception is impossible because cycles have
rank-width at most two.  Minimality gives $\rw(H)\leq2$, so
\zcref{prop:pivot-frontier} ensures that the pivot-equivalence class of $H$ occurs in the
complete parent frontier.

Now suppose that $13\leq n\leq16$.  Here $G$ is prime and
$3^{+2}$-rank-connected, and certainly has at least ten vertices, so
\zcref{thm:oum-chain} gives a prime $3^{+3}$-rank-connected pivot-minor $H$ of $G$ with
$n-1$ vertices.  Again $\rw(H)\leq2$.

If $n=13$, then \zcref{prop:pivot-frontier} puts the pivot-equivalence class of $H$ in
the complete frontier~$\cPiv_{12}$.  Rank connectivity is invariant under pivots, so the
$3^{+3}$-rank-connectivity filter retains this class.  Thus the pivot-equivalence class
of $H$ occurs in the parent set actually used for the 13-vertex search.

If $14\leq n\leq16$, then \zcref{prop:frontier-induction} puts the local-equivalence
class of $H$ in the complete frontier~$\cF_{n-1}$.  Rank connectivity is invariant under
local complementation, so the rank-connectivity filter selects this class, which is then
split exactly into pivot-equivalence and isomorphism classes as described above.  Thus
the pivot-equivalence class of $H$ occurs in the filtered parent set.

In each case, choose $\widehat G$ pivot-equivalent to $G$ such that
$\widehat G\setminus v$ is isomorphic to $H$.  Pivots and relabelings within this
induced subgraph lift to $\widehat G$, so a one-vertex extension of the chosen
representative is pivot-equivalent to $G$ and the generator reaches it.  For
$n\in\{15,16\}$, \zcref{prop:excluded-not-three-plus-one} says that $G$ is not
$3^{+1}$-rank-connected, whereas \zcref{prop:excluded-connectivity} says that it is
$3^{+2}$-rank-connected.  Consequently, a set of cut-rank at most two whose smaller
side has at least four but fewer than five vertices exists; its smaller side has exactly
four vertices.  The explicit four-set test in
\zcref{sec:further-optimizations} is therefore safe.

Every surviving candidate is tested to have rank-width greater than two and every proper
pivot-minor of rank-width at most two.  The exact pivot-equivalence key retains one
representative and cannot identify distinct classes.  Conversely, those direct tests prove that every
reported graph is an excluded pivot-minor.
\end{proof}

\begin{proof}[Proof of \zcref{thm:pivot-main}]
Graphs on at most five vertices have rank-width at most two.  By
\zcref{prop:pivot-completeness}, the search finds every excluded pivot-minor with at most
16 vertices.  Indexed by the number of vertices from 6 to 16, the verified class counts are
$0,0,2,447,146,10,4,0,0,0,0$, whose sum is 609.  No larger excluded pivot-minor exists
by \zcref{thm:excluded-pivot-bound}.  Full-key comparison proves that the catalog entries represent
distinct pivot-equivalence and isomorphism classes, and the independent verifier certifies
their rank-width and minimality.
\end{proof}

\section{Implementation and reproducibility}\label{sec:implementation}

The implementation uses nauty 2.9.3 \cite{MP2014}: dense nauty for ordinary
canonical labeling in the extension generator, and sparse nauty for local- and
pivot-equivalence classification
on the colored incidence graphs of \zcref{sec:lc-key}.

\subsection{Search results}

\begin{table}[H]
\centering
\caption{Verified prime frontier and excluded-vertex-minor counts.  Times are CPU times
summed over all chunks of the layer, covering the extension scan and the within-chunk
classification; the final cross-chunk key merge is not included.
The 16-vertex frontier
was not retained because no further layer is needed.}
\label{tab:counts}
\begin{tabular}{@{}r@{\qquad}r@{\qquad}r@{\qquad}r@{}}
\toprule
 & Prime $\rw\leq2$ & Excluded & CPU hours\\
$\abs{V}$ & local-equivalence classes & vertex-minor classes & (13--16 vertices)\\
\midrule
5  & 1         & 0  & \\
6  & 2         & 0  & \\
7  & 4         & 0  & \\
8  & 14        & 1  & \\
9  & 51        & 18 & \\
10 & 228       & 6  & \\
11 & 1,228     & 0  & \\
12 & 7,712     & 0  & \\
13 & 51,958    & 0  & 0.04\\
14 & 372,107   & 0  & 0.58\\
15 & 2,754,556 & 0  & 9.27\\
16 & not retained & 0 & 69.69\\
\bottomrule
\end{tabular}
\end{table}

The 15-vertex layer processed about $6.08\times10^9$ prime extensions and classified
about $7.81\times10^7$ raw rank-width-two candidates.  Its cross-chunk key merge reduced
$48{,}389{,}106$ within-chunk classes to $2{,}754{,}556$ global classes.  The 16-vertex
layer had the aggregate statistics in \zcref{tab:final-layer}.

\begin{table}[H]
\centering
\caption{Aggregate statistics for the 16-vertex excluded-vertex-minor-only layer.}
\label{tab:final-layer}
\begin{tabular}{@{}lr@{}}
\toprule
Quantity & Count\\
\midrule
15-vertex parents & 2,754,556\\
Prime one-vertex extensions & 90,134,581,432\\
Extensions of rank-width at most two & 639,299,046\\
Nonminimal extensions of rank-width greater than two & 89,495,282,386\\
Cached deletion-test hits & 45,011,112,024\\
Excluded vertex-minors & 0\\
\bottomrule
\end{tabular}
\end{table}

The reported computation used ten concurrent jobs on a 10-core Apple M2 Pro.

\subsection{Pivot-minor search results}

The pivot-minor computation produced the counts in \zcref{tab:pivot-counts}.

\begin{table}[H]
\centering
\caption{Verified excluded pivot-minor classes by the number of vertices.}
\label{tab:pivot-counts}
\begin{tabular}{@{}r@{\qquad}r@{\qquad}r@{\qquad}r@{}}
\toprule
$\abs{V}$ & Classes & $\abs{V}$ & Classes\\
\midrule
6  & 0   & 12 & 4\\
7  & 0   & 13 & 0\\
8  & 2   & 14 & 0\\
9  & 447 & 15 & 0\\
10 & 146 & 16 & 0\\
11 & 10  & Total & 609\\
\bottomrule
\end{tabular}
\end{table}

For the last parent layer, $28{,}911$ relevant 15-vertex local-equivalence classes split
into $15{,}138{,}031$ pivot-equivalence classes.  The large-layer statistics are shown in
\zcref{tab:pivot-large-layers}.  In each row, the first column is the number of vertices
in the generated one-vertex extensions; their parent graphs have one fewer vertex.  We
partitioned the parent list into the indicated number of disjoint \emph{chunks}, each
processed as an independent, restartable work unit.  The parent counts, extension counts,
and CPU hours are totals over all chunks; CPU time is measured as in \zcref{tab:counts}.

\begin{table}[H]
\centering
\caption{Large-layer pivot-minor computation statistics.}
\label{tab:pivot-large-layers}
\begin{tabular}{@{}rrrrr@{}}
\toprule
Generated graph $\abs{V}$ & Chunks & Parent graphs & Prime extensions & CPU hours\\
\midrule
13 & 32  & 263,343    & 1,068,909,237   & 0.30\\
14 & 128 & 1,169,230  & 9,531,562,960   & 3.42\\
15 & 320 & 4,694,158  & 76,707,235,878  & 16.63\\
16 & 512 & 15,138,031 & 495,346,650,382 & 107.35\\
\bottomrule
\end{tabular}
\end{table}

The graph lists and verification software are available in the
public repository\footnote{\url{https://github.com/sangiloum/rwd2}}.

\section{The excluded vertex-minors}\label{sec:figures}

The following plates show the 25 representatives used in \zcref{thm:main}.
Straight edges may cross; crossings are not vertices.  Each drawing is an unlabeled
realization of the corresponding graph6 record in \zcref{app:graph6}.

\begin{figure}[H]
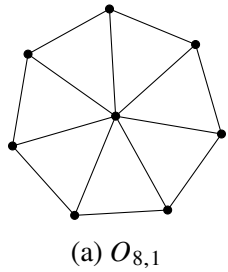

\centering
\begin{subfigure}[t]{0.50\textwidth}
  \centering
\obstructiongraph{1}
\caption{$O_{8,1}$}
\end{subfigure}
\caption{The unique excluded-vertex-minor class on eight vertices, displayed as the wheel graph.}
\label{fig:o8}
\end{figure}

\begin{figure}
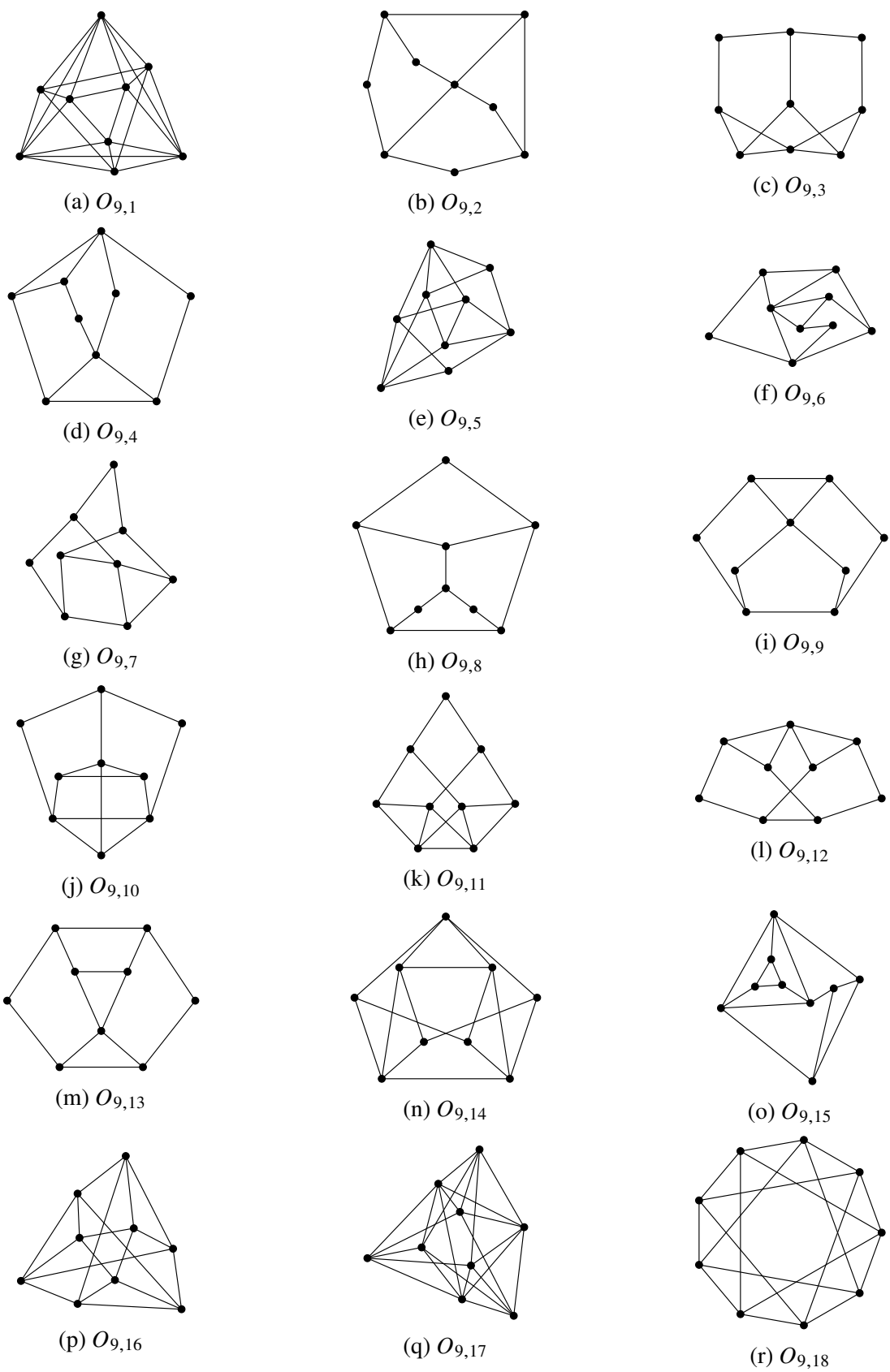

\centering
\obstructionpanel{2}{O_{9,1}}\hfill
\obstructionpanel{3}{O_{9,2}}\hfill
\obstructionpanel{4}{O_{9,3}}

\smallskip
\obstructionpanel{5}{O_{9,4}}\hfill
\obstructionpanel{6}{O_{9,5}}\hfill
\obstructionpanel{7}{O_{9,6}}

\smallskip
\obstructionpanel{8}{O_{9,7}}\hfill
\obstructionpanel{9}{O_{9,8}}\hfill
\obstructionpanel{10}{O_{9,9}}

\smallskip
\obstructionpanel{11}{O_{9,10}}\hfill
\obstructionpanel{12}{O_{9,11}}\hfill
\obstructionpanel{13}{O_{9,12}}

\smallskip
\obstructionpanel{14}{O_{9,13}}\hfill
\obstructionpanel{15}{O_{9,14}}\hfill
\obstructionpanel{16}{O_{9,15}}

\smallskip
\obstructionpanel{17}{O_{9,16}}\hfill
\obstructionpanel{18}{O_{9,17}}\hfill
\obstructionpanel{19}{O_{9,18}}
\caption{The 18 excluded-vertex-minor classes on nine vertices.}
\label{fig:o9}
\end{figure}

\begin{figure}
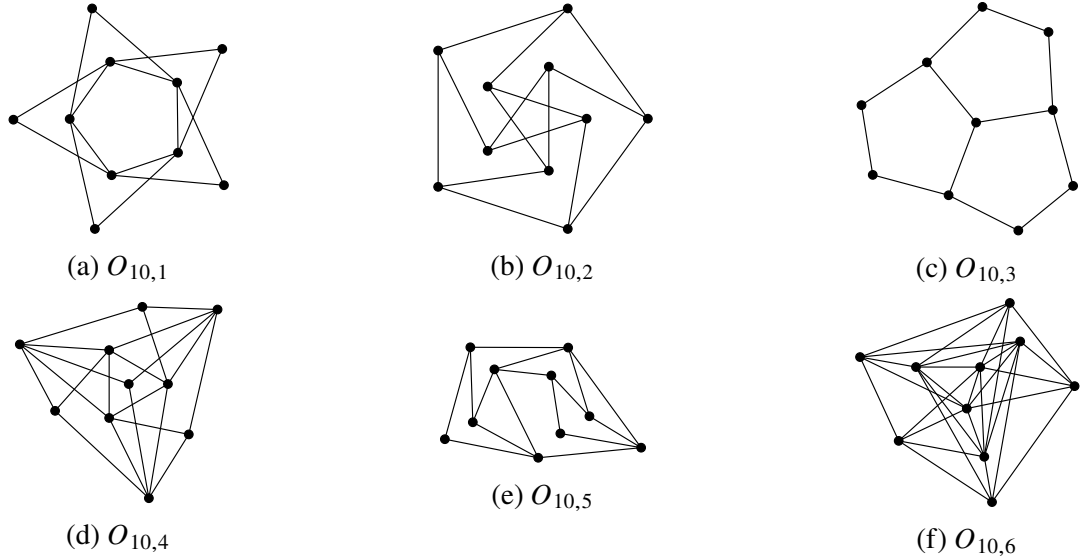

\centering
\obstructionpanel{20}{O_{10,1}}\hfill
\obstructionpanel{21}{O_{10,2}}\hfill
\obstructionpanel{22}{O_{10,3}}

\smallskip
\obstructionpanel{23}{O_{10,4}}\hfill
\obstructionpanel{24}{O_{10,5}}\hfill
\obstructionpanel{25}{O_{10,6}}
\par\vspace{1.5em}
\caption{The six excluded-vertex-minor classes on ten vertices; $O_{10,2}$ is the Petersen graph.}
\label{fig:o10}
\end{figure}

\section{Concluding remarks}

The significance of the computation lies not only in the two catalogs, but also in three
methodological contributions that made the computation possible.

First, this work completes a search that the author began in 2007 with a C program.  The
earlier approach used \textsc{geng} from nauty to generate graphs satisfying
prescribed preliminary conditions and then tested them for the excluded-minor property.
For graphs on 16 vertices, however, this still required searching such a large ambient
family that extrapolated running times ranged from months to years, and the computation
was not completed.  The present approach reverses the search: guided by the prime
one-vertex reduction theorems, it generates
one-vertex extensions only from the prime rank-width-two frontier.  Thus it searches the
boundary of the class rather than a large portion of the space of all graphs, making the
complete computation feasible on a laptop.

Second, we give exact canonical encodings of graphs up to local equivalence and up to
pivot equivalence.  Inspired by code-graph constructions from coding theory
\cite{Ostergard2002,DP2009} and expressed through the isotropic-system representation of
a graph, these encodings assign the same canonical key precisely to graphs in the same
equivalence-and-isomorphism class.  They therefore allow the search to identify
local-equivalence classes, and the finer pivot-equivalence classes, by exact key comparison
instead of repeated traversal of equivalence orbits.

Third, we give a practical exact algorithm for deciding whether a graph has rank-width at
most two.  Its block-merging rules and the titanic-partition argument yield both a
reconstructible correctness proof and the speed needed for the enormous number of
rank-width tests in the search.  A general subset dynamic program remains useful as an
independent verifier, but would be far too slow as the main membership test.

Together, these ideas suggest a reusable framework for other vertex-minor-closed classes:
an exact and practical membership oracle, a one-vertex reduction theorem for the relevant
prime objects, an exact canonical encoding of the chosen equivalence relation, and a
finite bound on excluded minors.  The resulting frontier may still grow rapidly, but the
framework separates the mathematical completeness argument from parallelization and
other search-specific optimizations.

\bigskip
\noindent\textbf{Statement of AI use.}
The author used OpenAI's GPT-5.6 Sol and Anthropic's Claude Fable 5 for programming
assistance, monitoring computations, drafting, and editorial suggestions.  The author
independently verified the mathematical arguments, source code, computations, and final
text and takes full responsibility for the contents.

\appendix
\section{Graph6 records}\label{app:graph6}

\zcref{tab:graph6} is the complete machine-readable list of the displayed vertex-minor
representatives, and \zcref{tab:pivot-graph6} gives the complete list of excluded
pivot-minors.  The records use the standard graph6 upper-triangle order from nauty.  The edge counts
in \zcref{tab:graph6} provide short transcription checks.  In
\zcref{tab:pivot-graph6}, records are grouped by the number of vertices and read from left
to right, row by row.

\begin{longtable}{@{}c@{\quad}r@{\quad}r@{\quad}l@{\qquad}c@{\quad}r@{\quad}r@{\quad}l@{}}
\caption{The 25 symmetry-enhanced excluded-vertex-minor representatives.}\label{tab:graph6}\\
\toprule
Graph & $|V|$ & $|E|$ & graph6 record &
Graph & $|V|$ & $|E|$ & graph6 record\\
\midrule
\endfirsthead
\multicolumn{8}{c}{\tablename~\thetable\ (continued)}\\
\toprule
Graph & $|V|$ & $|E|$ & graph6 record &
Graph & $|V|$ & $|E|$ & graph6 record\\
\midrule
\endhead
\prettyobstructionrows
\bottomrule
\end{longtable}

\clearpage
\begingroup
\fontsize{7}{7.6}\selectfont
\setlength{\tabcolsep}{1.8pt}
\renewcommand{\arraystretch}{0.78}
\begin{longtable}{@{}llllllllll@{}}
\caption{The graph6 records of the 609 excluded pivot-minors.}
\label{tab:pivot-graph6}\\
\toprule
\endfirsthead
\multicolumn{10}{c}{\tablename~\thetable\ (continued)}\\
\toprule
\endhead
\pivotobstructionrows
\bottomrule
\end{longtable}
\endgroup

\bibliographystyle{abbrv}

\end{document}